%% file: main.tex
\documentclass{article}

\usepackage[a4paper]{geometry}		% Formato del foglio
\usepackage[english]{babel}			% Supporto per l'italiano
\usepackage[utf8]{inputenc}			% Supporto per UTF-8
\usepackage[a-1b]{pdfx}			% File conforme allo standard PDF-A (obbligatorio per la consegna)
\usepackage{fancyhdr}

\usepackage{enumerate}

\usepackage{graphicx}				% Funzioni avanzate per le immagini
\usepackage{hologo}					% Bibtex logo with \hologo{BibTeX}
\usepackage{xcolor}				% Gestione avanzata dei colori

\usepackage{amssymb,amsmath,amsthm,bbm} % Simboli matematici
\usepackage{mathrsfs}
\usepackage{listings}	
\usepackage{mathtools}
\hypersetup{colorlinks,allcolors=blue}

\newtheorem{Theorem}{Theorem}[section]
\newtheorem{Definition}{Definition}[section]
\newtheorem{Proposition}{Proposition}[section]
\newtheorem{Corollary}{Corollary}[section]
\newtheorem{Lemma}{Lemma}[section]
\newtheorem{Remark}{Remark}[section]

\newtheorem*{specialAssumption}{Assumption}
\numberwithin{equation}{section}

\makeatletter
\renewenvironment{proof}[1][\proofname] {\par\pushQED{\qed}\normalfont\topsep6\p@\@plus6\p@\relax\trivlist\item[\hskip\labelsep\itshape\bfseries#1\@addpunct{.}]\ignorespaces}{\popQED\endtrivlist\@endpefalse}
\makeatother

\def \N{\mathbb{N}}
\def \R{\mathbb{R}}

\def \E{\mathbb{E}}
\def \F{\mathbb{F}}

\def \P{\mathbb{P}}

\def \S{\mathbb{S}}

\def \Cc{{\cal C}}

\def \Fc{{\cal F}}
\def \Gc{{\cal G}}
\def \Hc{{\cal H}}

\def \Lc{{\cal L}}
\def \Pc{{\cal P}}

\def \Sc{{\cal S}}
\def \Tc{{\cal T}}

\def \Wc{{\cal W}}

\DeclareMathOperator*{\esssup}{ess\,sup}

\newcommand{\Norm}[1]{ \lVert {#1} \rVert}

\title{Mean field optimal stopping with uncontrolled state}
\author{
Andrea COSSO\footnote{Universit\`a degli Studi di Milano; andrea.cosso@unimi.it} \quad\qquad
Laura PERELLI\footnote{Universit\`a degli Studi di Milano; laura.perelli@unimi.it\vspace{2mm}\\\textbf{Acknowledgements.} A. Cosso acknowledges support from GNAMPA-INdAM (of which L. Perelli is also member), the MUR project PRIN 2022 ``Entropy martingale optimal transport and McKean-Vlasov equations'', and the MUR project PRIN 2022 PNRR ``Probabilistic methods for energy transition''.}}
\date{March 6, 2025}

\allowdisplaybreaks

\begin{document}
\maketitle

%sections

\begin{abstract}
\noindent We study a specific class of finite-horizon mean field optimal stopping problems by means of the dynamic programming approach. In particular, we consider problems where the state process is not affected by the stopping time. Such problems arise, for instance, in the pricing of American options when the underlying asset follows a McKean-Vlasov dynamics. Due to the time inconsistency of these problems, we provide a suitable reformulation of the original problem for which a dynamic programming principle can be established. To accomplish this, we first enlarge the state space and then introduce the so-called extended value function. We prove that the Snell envelope of the original problem can be written in terms of the extended value function, from which we can derive a characterization of the smallest optimal stopping time. On the enlarged space, we restore time-consistency and in particular establish a dynamic programming principle for the extended value function. Finally, by employing the notion of Lions measure derivative, we derive the associated Hamilton-Jacobi-Bellman equation, which turns out to be a second-order variational inequality on the product space $[0,T]\times\mathbb R^d\times\mathcal P_2(\mathbb R^d)$; under suitable assumptions, we prove that the extended value function is a viscosity solution to this equation.
\end{abstract}

\vspace{5mm}


\vspace{5mm}


\input{intro}

\input{formulation}
\input{reformulation}
\input{snell_envelope}
\input{dpp}
\input{hjb}

%bibliography

\small
\bibliographystyle{plain}
\bibliography{bibliografia.bib}

\end{document}

%% file: intro.tex
\section{Introduction}

In the present article, we study a class of mean field optimal stopping problems, in which the state process evolves according to the following McKean-Vlasov stochastic differential equation:
\begin{equation}\label{SDE_Intro}
\begin{cases}
        dX_s = b(s,X_s,\P_{X_s})ds + \sigma(s,X_s,\P_{X_s})dW_s, &\quad s\in[t,T],\\
        X_t = \xi.
\end{cases}
\end{equation}
Here $\P_{X_s}$ denotes the law of $X_s$, $W$ is an $m$-dimensional Brownian motion on a complete probability space $(\Omega,\Fc,\P)$, $\xi$ is a square integrable random variable taking values in $\R^d$, and the coefficients $b$ and $\sigma$ satisfy standard Lipschitz and linear growth conditions (see Assumption (\nameref{A1}); the precise setting will be described in the first paragraph of Section \ref{S:OptStoppProblem}). The underlying filtration $\F=(\Fc_t)_{t\in[0,T]}$ is the $\P$-completion of the filtration generated by $W$ and an independent $\sigma$-algebra $\Gc$; the presence of $\Gc$ is relevant at time $t=0$ to ensure that $\Fc_0$ is not trivial, and consequently that the initial condition $\xi$ at time $t=0$ is not necessarily a constant, but can have any finite second-order moment distribution. Denoting by $X^{t,\xi}=(X_s^{t,\xi})_{s\in[t,T]}$ the solution to equation \ref{SDE_Intro} (we refer to Section \ref{S:OptStoppProblem} for the study of equation \eqref{SDE_Intro} and properties of its solution), the goal is to maximize the following reward functional with respect to the family $\Tc_{t,T}$ of all $\F$-stopping times that take values in the interval $[t,T]$:
\[
    J(t,\xi,\tau) \ = \ \E\bigg[\int_{t}^{\tau} f\big(s,X^{t,\xi}_s,\P_{X^{t,\xi}_s}\big)ds \ + \ g\big(X^{t,\xi}_\tau\big)\bigg].
\]
The value function is then given by
\[
    V(t,\xi) \ = \ \sup_{\tau\in\Tc_{t,T}} J(t,\xi,\tau).
\]
Our aim is to address the above optimal stopping problem using the dynamic programming method, and also to obtain an explicit description of the Snell envelope, from which one can directly deduce the optimal stopping rule. The main difficulty of such a problem is its time-inconsistency, meaning that $V$ does not satisfy a dynamic programming principle, and therefore does not solve a Hamilton-Jacobi-Bellman equation. In order to address this issue, inspired by \cite[Section 6.5]{CD18_I}, where the classical mean field control problem is studied, in Section \ref{reformulation} we formulate the mean field optimal stopping problem on an enlarged space. More precisely, we introduce a new state variable, solving the following stochastic differential equation:
\begin{equation}\label{SDE_Intro_Reform}
\begin{cases}
        dX_s = b(s,X_s,\P^{t,\mu}_s)ds + \sigma(s,X_s,\P^{t,\mu}_s)dW_s, & s\in[t,T],\\
        X_t = x\in\R^d,
    \end{cases}
\end{equation}
where $\mu$ and $\P_s^{t,\mu}$ are respectively the law of $\xi$ and $X_s^{t,\xi}$, which are the initial condition and the solution to \eqref{SDE_Intro}. Note that equation \eqref{SDE_Intro_Reform} is standard, meaning that the coefficients do not depend on the marginal laws of the solution, but on the marginal laws of the known process $X^{t,\xi}$. Denoting by $X^{t,x,\mu}=(X_s^{t,x,\mu})_{s\in[t,T]}$ the solution to equation \eqref{SDE_Intro_Reform}, the new state process is the pair $(X_s^{t,x,\mu},\P_s^{t,\mu})_{s\in[t,T]}$. The presence of the additional process $X^{t,x,\mu}$ allows us to decouple the position of the state process and its law, which are the two relevant variables in the original problem. We can then formulate the mean field optimal stopping problem on such an enlarged state space. We refer to this new problem as the extended mean field optimal stopping problem. More precisely, we denote by $\tilde\Tc_{t,T}$ the family of $[t,T]$-valued stopping times with respect to the $\P$-completion of the filtration generated by the Brownian increments $(W_s-W_t)_{s\in[t,T]}$. We then consider the reward functional
\[
\tilde J(t,x,\mu,\tau) \ = \ \E\bigg[\int_{t}^{\tau} f(s,X^{t,x,\mu}_s,\P^{t,\mu}_s)ds \ + \ g(X^{t,x,\mu}_\tau)\bigg]
\]
and the so-called extended value function
\[
            \tilde V(t,x,\mu) \ = \ \sup_{\tau\in\tilde\Tc_{t,T}} \tilde J(t,x,\mu,\tau).
\]
Since $\Fc^{t,W}_s\subset\Fc_s$, for every $s\in[t,T]$, it follows that $\tilde\Tc_{t,T}\subset\Tc_{t,T}$. Notice that formulating the extended problem over $\tilde\Tc_{t,T}$ is not only important for achieving a fundamental result (Theorem \ref{T:representation V tilde}; see also Remark \ref{R: why Ttilde}), but it is also consistent with the reformulation of our problem; indeed, while in the original problem the evolution of the state process $X^{t,\xi}$ depends on two independent random components, namely $\xi$ and $(W_s-W_t)_{s\in[t,T]}$, in the extended problem the only source of noise is provided by the Brownian increments. Despite this, a posteriori, we are able to show that we obtain the same $\tilde V$ even when taking the supremum over $\Tc_{t,T}$, see Corollary \ref{C: enlarge stopping times for Vtilde}.
\noindent The first result in this paper, which turns out to be crucial, establishes that $\Tilde{V}$ satisfies the following relation (Theorem \ref{T:representation V tilde}):
\[
\tilde V(t,\xi,\mu) \ = \ \esssup_{\tau\in\Tc_{t,T}} \E\bigg[\int_{t}^{\tau} f\big(s,X^{t,\xi}_s,\P_{X_s^{t,\xi}}\big)ds \ + \ g\big(X^{t,\xi}_\tau\big) \bigg|\Fc_t\bigg], \qquad \P\text{-a.s.}
\]
Notice that the set of stopping times considered is the largest one. Based on this representation result, we derive two fundamental consequences. First, from the characterization of the Snell envelope, we deduce that the smallest optimal stopping time $\hat{\tau}^{t,\xi}$ of the original problem is given by (Theorem \ref{T:Snell})
\[
    \hat{\tau}^{t,\xi} \ = \ \inf\big\{s\in[t,T] \ \big| \ \Tilde{V}\big(s,X^{t,\xi}_s,\P_{X_s^{t,\xi}}\big) \, = \, g\big(X^{t,\xi}_s\big)\big\}.
\]
On the other hand, we show that the extended mean field optimal stopping problem is indeed time-consistent. In fact, we are able to prove that $\Tilde{V}$ satisfies a dynamic programming principle (Theorem \ref{T:DPP MKV 1}) and, furthermore, that it solves the Hamilton-Jacobi-Bellman equation \eqref{HJB MKV} on $[0,T]\times\R^d\times\mathcal{P}_2(\R^d)$  in the viscosity sense (Theorem \ref{T: V tilde solution HJB}), where $\Pc_2(\R^d)$ is the Wasserstein space of probability measures on $\R^d$ with finite second moment. This equation involves the so-called $L$-derivatives or Lions derivatives, see for instance \cite[Chapter 5]{CD18_I}. It is worth noting that a particularly relevant research topic is the uniqueness of viscosity solutions for second-order equations on the Wasserstein space; however, for equation \eqref{HJB MKV}, which as expected turns out to be a variational inequality on $[0,T]\times\R^d\times\Pc_2(\R^d)$, such an uniqueness result is still missing. Another important consequence of Theorem \ref{T:representation V tilde} is a ``disintegration'' result, which shows that the original value function $V$ can be written in terms of $\Tilde{V}$ (Corollary \ref{C:relation V - V tilde}):
\[
    V(t,\xi) \ = \ \E\big[\tilde V(t,\xi,\mu)\big] \ = \ \int_{\R^d}\tilde V(t,x,\mu)\,\mu(dx),
\]
with $\xi$ having distribution $\mu$. This result, along with the formula for the optimal stopping time $\hat{\tau}^{t,\xi}$, highlights a strong connection between the original and extended problems. Indeed, by solving the extended problem (which, as explained above, can be studied by the dynamic programming method and characterized in terms of a suitable dynamic programming equation), we can obtain both the value function and the optimal stopping time for the original problem.\\
\\
The literature on mean field optimal stopping problems has still few contributions, all very recent. We mention in particular \cite{TTZ1}, along with its companion papers \cite{TTZ2,TTZ3}, in which a different mean field optimal stopping problem is studied, characterized by a ``controlled'' state process (in contrast to the present work, where, as specified in the title, the state process is uncontrolled). More precisely, in \cite{TTZ1} the stopped state process is considered:
\begin{equation}\label{SDE_TTZ}
X_s \ = \ \xi + \int_t^{s\wedge\tau} b(r,X_r,\P_{X_r})dr + \int_t^{s\wedge\tau}\sigma(r,X_r,\P_{X_r})dW_r, \qquad s\in[t,T].
\end{equation}
It is easy to see that \eqref{SDE_Intro} and \eqref{SDE_TTZ} generally have different solutions (notice, in fact, that the solution to \eqref{SDE_TTZ} is always constant after $\tau$). To better understand the difference between \eqref{SDE_Intro} and \eqref{SDE_TTZ}, it is useful to think about the relationship between mean field control problems and stochastic differential games with many symmetrically interacting players (for a detailed introduction to this topic and further details, see for instance \cite{Lions,Carda13,CD18_I}). More precisely, mean field control problems, as well as mean field games, can be thought as limiting problems, when the number of players goes to infinity, of stochastic differential games where the interaction between the various state processes happens through the empirical distribution. The mean field game corresponds to the case where Nash equilibrium is considered in the game, while mean field control problem is obtained by considering the Pareto equilibrium; the latter is also referred to as the social planner problem. The mean field optimal stopping problem considered in the present article, as well as the one considered in \cite{TTZ1}, are mean field control problems, so both refer to the social planner problem. If in the empirical distribution the stopped state processes are considered, then in the limit, when the number of players tends to infinity, the problem studied in \cite{TTZ1} with state equation \eqref{SDE_TTZ} is obtained, as shown rigorously in \cite{TTZ3}. On the other hand, although the convergence issue is beyond the scope of this work, it is reasonable to expect that equation \eqref{SDE_Intro} corresponds to the case where the empirical distribution is made of unstopped state processes, that is, it models a situation where the social planner %has a reward when he/she stops a state process, however he/she
cannot act on a state in order to stop its evolution. Notice however that the mean field optimal stopping problem with uncontrolled state studied in the present paper, with state process \eqref{SDE_Intro}, is relevant not only because of its potential connection with the social planner problem, but also whenever it is useful to consider an optimal stopping problem in which the state dynamics is of McKean-Vlasov type and the state process cannot be stopped by the controller. This is the case, for instance, in the pricing of American options with an underlying asset following a McKean-Vlasov dynamics. As suggested in \cite{YY}, an example of underlying asset whose behavior can be adequately described by a McKean-Vlasov type dynamics could be an ETF (Exchange Traded Fund), as it is determined by a basket of stocks, those that make up the corresponding index. For further details, see \cite[Section 5.2]{YY}, where the authors derive the dynamics of an ETF in a discrete-time model, starting from the price dynamics of each component stock of the corresponding index, and then taking the limit as the number of stocks goes to infinity. As noticed in \cite{YY}, American options on ETFs are financial instruments particularly useful for hedging systemic risk, such as the downward risk of an index; furthermore, the trading volume of ETFs, as well as of their options, has dramatically grown in the last decades. In the literature, other works studying mean field optimal stopping problems are \cite{AO}, \cite{BS20}, \cite{DM}. In all of them, as in the present article, an uncontrolled state process is considered. In particular, \cite{DM} studies a specific class of problems where the dependence on the law is through the expected value, and also addresses the case of recursive utility; in \cite{DM} such a problem is solved by approximating the corresponding value process with a sequence of Snell envelopes, therefore obtaining an explicit characterization of the optimal stopping rule. In \cite{AO}, optimal stopping problems for conditional McKean-Vlasov jump diffusions are studied, providing a verification theorem and applying it to the problem of finding the optimal time to sell in a market with common noise and jumps. Finally, \cite{BS20} investigates numerical methods for mean field optimal stopping problems, for which we also refer to \cite{YY}, where optimal stopping for mean field Markov decision processes is studied. We further mention \cite{BEH18,DEH23}, where a topic closely related to mean field optimal stopping problems is investigated, namely mean field backward stochastic differential equations. Finally, we mention the papers \cite{Bertucci,BouveretDumitrescuTankov,CarmonaDelarueLacker,DumitrescuLeutscherTankov,Nutz,WangZhou} which are devoted to the mean field game of optimal stopping.\\
\\
\noindent The rest of the paper is organized as follows. In Section \ref{S:OptStoppProblem} we formulate the mean field optimal stopping problem, stating the assumptions on the coefficients, and discussing some properties of the state process and the value function. Section \ref{reformulation} is devoted to the extended mean field optimal stopping problem; in particular, we introduce the new state process and the extended value function; we also prove the fundamental Theorem \ref{T:representation V tilde}, as well as its two direct consequences, namely Corollaries \ref{C: enlarge stopping times for Vtilde} and \ref{C:relation V - V tilde}. In Section \ref{S:Snell}, we show that the Snell envelope of the original problem can be expressed through the extended value function $\tilde V$; from this result, we derive a characterization of the smallest optimal stopping rule in terms of $\tilde V$. Section \ref{S:DPP} is devoted to the proof of the time-consistency for the extended problem, namely the dynamic programming principle for $\tilde V$. Finally, in Section \ref{S:HJB}, we characterize the extended value function in terms of the Hamilton-Jacobi-Bellman equation \eqref{HJB MKV}, which turns out to be a variational inequality on $[0,T]\times\R^d\times\Pc_2(\R^d)$; we conclude Section \ref{S:HJB} by proving that $\tilde V$ is a viscosity solution of this equation, Theorem \ref{T: V tilde solution HJB}.

%% file: formulation.tex
\section{The mean field optimal stopping problem}\label{S:OptStoppProblem}
\textbf{Probabilistic framework.} We fix a time horizon $T\in(0,+\infty)$ and a complete probability space $(\Omega,\Fc,\P)$ on which an $m$-dimensional Brownian motion $W = (W^1_t,\dots,W^m_t)_{t\in[0,T]}$ is defined. We equip the probability space with the filtration $\F = (\Fc_t)_{t\in[0,T]}$ defined by $\Fc_t:=\Fc^W_t\vee\Gc$, $t\in[0,T]$, where $\F^W=(\Fc^W_t)_{t\in[0,T]}$ is the $\P$-completion of the natural filtration of $W$, while $\Gc$ is a sub-$\sigma$-algebra of $\Fc$ independent of $\Fc^W_T$ and such that there exists a $\Gc$-measurable random variable $U$ having uniform distribution on $[0,1]$ (see \cite[Lemma 2.1]{CossoMartini} for equivalent conditions to the latter property). Notice that $\F$ satisfies the usual conditions. Moreover, given $t\in[0,T]$ fixed, we denote by:
\begin{itemize}
    \item $\Tc_{t,T}$ the set of $\F$-stopping times with values in $[t,T]$;
    \item $\S^2_{t,T}$ the set of continuous and $\F$-progressive processes $X:[t,T]\times\Omega\longrightarrow\R^d$ such that 
    \[
        \|X\|_{ \S^2_{t,T}}:=\E\bigg[\sup_{t\leq s\leq T}|X_s|^2\bigg]^{1/2}<+\infty;
    \]
    %\textcolor{blue}{ where $|\cdot|$ is the Euclidean norm in $\R^d$;}
    \item $(\mathcal{P}_2(\R^d),\mathcal{W}_2)$ the second-order Wasserstein space, that is the set of all probability measures on the Borel subsets of $\R^d$ with finite second moment, endowed with the 2-Wasserstein distance $\Wc_2$. We recall that, for every $\mu,\nu\in\mathcal{P}_2(\R^d)$,
    \begin{align*}
        \Wc_2^2(\mu,\nu) \coloneqq \inf\bigg\{\int_{\R^d\times\R^d}|x-y|^2\pi(dx,dy) \ \Big| \ \pi\in\mathcal{P}(&\R^d\times\R^d)\text{ such that}\\ &\pi(\,\cdot\times\R^d) = \mu,\, \pi(\R^d\times\cdot\,) = \nu\bigg\}.
    \end{align*}
    We also denote
    \begin{equation}\label{norm_2}
         \Norm{\mu}_2 \ \coloneqq \ \Wc_2(\mu,\delta_0) \ = \ \bigg(\int_{\R^d} |x|^2\mu(dx)\bigg)^{\frac{1}{2}};
    \end{equation}
    \item $L^2(\Fc_t;\R^d)$ the set of square-integrable, $\Fc_t$-measurable random variables taking values in $\R^d$. In particular, for every $\xi\in L^2(\Fc_t;\R^d)$ we denote its $L^2$-norm by $\Norm{\xi}_{L^2}$.
\end{itemize}
Now consider the measurable functions $b,\sigma,f\colon[0,T]\times\R^d\times\mathcal{P}_2(\R^d)\rightarrow\R^d,\R^{d\times m},\R$ and $g\colon\R^d\rightarrow\R$. Throughout the paper the following assumptions will always be in force.
\begin{specialAssumption}[\textbf{A$_{b,\sigma}$}]\label{A1}
%\leavevmode\begin{itemize}
\quad\begin{enumerate}[\upshape1)]
    %\item $b,\sigma$ are measurable functions;
    \item $b,\sigma$ are Lipschitz continuous in $(x,\mu)$ uniformly in $t$: there exists $L>0$ such that, for all  $t\in[0,T],\, x,y\in\R^d,\, \mu,\nu\in\mathcal{P}_2(\R^d)$,
    \[
        |b(t,x,\mu)-b(t,y,\nu)| \ + \ |\sigma(t,x,\mu)-\sigma(t,y,\nu)| \ \leq \ L(|x-y| \ + \ \Wc_2(\mu,\nu));
    \]
    \item there exists $K>0$ such that, for all $t\in[0,T]$,
    \[
        |b(t,0,\delta_0)| \ + \ |\sigma(t,0,\delta_0)| \ \leq \ K.
    \]
\end{enumerate}
\end{specialAssumption}

\quad

\begin{specialAssumption}[\textbf{A$_{f,g}$}]\label{A2}
%\leavevmode\begin{itemize}
\quad\begin{enumerate}[\upshape1)]
    %\item [1)] $f,g$ are measurable functions;
    \item $f$ is locally uniformly continuous in $(x,\mu)$ uniformly in $t$: for every $\varepsilon>0$ and $n\in\N$, there exists $\delta = \delta(\varepsilon,n)>0$ such that, for every $t\in[0,T],\, (x,\mu),(y,\nu)\in\R^d\times\mathcal{P}_2(\R^d)$ with $|x|+\Norm{\mu}_2
    \leq n$, $|y|+\Norm{\nu}_2\leq n$
    \[
        |x-y| \ + \ \Wc_2(\mu,\nu) \ \leq \ \delta \quad \Longrightarrow \quad |f(t,x,\mu) \ - \ f(t,y,\nu)| \ \leq \ \varepsilon;
    \]
    \item $g$ is continuous;   
    \item $f$ has sub-quadratic growth in $(x,\mu)$ uniformly in $t$ and $g$ has sub-quadratic growth in $x$: there exists $K>0$ such that, for every $t\in[0,T]$, $x\in\R^d$, $\mu\in\mathcal{P}_2(\R^d)$,
    \[
        |f(t,x,\mu)| \ \leq \ K(1 + |x|^2 + \Norm{\mu}_2^2), \qquad |g(x)| \ \leq \ K(1 + |x|^2).
    \]
    \end{enumerate}
\end{specialAssumption}
\begin{Remark}\label{R: regularity assumptions}
    \textnormal{Assumptions (\nameref{A1})-(\nameref{A2})} are as in \cite{cosso1}, except for $g$ that here does not depend on $\mu$. This is crucial for instance at the end of Step 1.1 in the proof of Theorem \ref{T:representation V tilde}.
\end{Remark}

\noindent\textbf{State process and its properties.} For every $t\in[0,T]$ and $\xi\in L^2(\Fc_t;\R^d)$, the state process evolves according to the following McKean-Vlasov stochastic differential equation:
\begin{equation}\label{MKV SDE}
    \begin{cases}
        dX_s = b(s,X_s,\P_{X_s})ds + \sigma(s,X_s,\P_{X_s})dW_s, & s\in[t,T],\\
        X_t = \xi.
    \end{cases}
\end{equation}
We observe that, due to the definition of $\F$, $\xi$ is always independent of the Brownian motion's increments, which ensures that the stochastic integral is well-defined. Moreover, thanks to the choice of $\Gc$, the law of $\xi$ can be any element of $\mathcal{P}_2(\R^d)$, as it follows for instance from \cite[Lemma 2.1]{CossoMartini}.\\ % \textcolor{green}{(L: bisogna citare qualcosa per questo fatto?)}\\
Under Assumption (\nameref{A1}) we have existence and uniqueness of  the state process for any choice of initial data; in fact, the following proposition holds.
\begin{Proposition}\label{T:ex-uniq MKV}
    For every $t\in[0,T]$ and $\xi\in L^2(\Fc_t;\R^d)$, there exists a unique (up to $\P$-indistinguishability) process $X^{t,\xi}\in\S^2_{t,T}$ solution to \eqref{MKV SDE} that satisfies the following estimate: there exists a constant $C>0$ such that
%    \begin{align}
    \[
        \Norm{X^{t,\xi}}_{\S^2_{t,T}}^2 \ \leq \ C(1 + \Norm{\xi}_{L^2}^2). %\label{estimate MKV 1} \\
    \]
\end{Proposition}
\begin{proof}
See \cite[Theorem 4.21 and Section 5.7.4]{CD18_I}.
\end{proof}
\begin{Remark}%\label{R: cont marginal flow}
    Since the state process $X^{t,\xi}$ is in $\S_{t,T}^2$, the flow of its marginal distributions is a continuous flow of probability measures in $\mathcal{P}_2(\R^d)$, that is the map $(\P_{X^{t,\xi}_s})_{s\in[t,T]}\colon[t,T]\rightarrow\mathcal{P}_2(\R^d)$ is continuous. 
\end{Remark}

\noindent We now prove that the marginals of the state process are invariant with respect to the initial law (see Remark \ref{R: same marginals}). This invariance property will be crucial in reformulating the problem in Section \ref{reformulation}.
\begin{Proposition}\label{P: same marginal}
    Let $t\in[0,T]$ and $\xi,\eta\in L^2(\Fc_t;\R^d)$ with $\P_\xi = \P_\eta = \mu$. Then the respective state processes $X^{t,\xi}$, $X^{t,\eta}$ have the same marginal distributions, i.e.
    \[
        \P_{X^{t,\xi}_s} \ = \ \P_{X^{t,\eta}_s},\qquad\text{for all } s\in[t,T].
    \]
\end{Proposition}
\begin{proof}
    Define the two following sequences of processes:
    \[
        \begin{cases}
            X^0_s := \xi, & s\in[t,T]\\
            X^{m+1}_s := \xi + \int_t^s b(r,X^m_r,\P_{X^m_r})dr + \int_t^s \sigma(r,X^m_r,\P_{X^m_r})dW_r, & s\in[t,T],\,m\geq0
        \end{cases}
    \]
    \[
        \begin{cases}
            \tilde X^0_s := \eta, & s\in[t,T]\\
            \tilde X^{m+1}_s := \eta + \int_t^s b(r,\tilde X^m_r,\P_{\tilde X^m_r})dr + \int_t^s \sigma(r,\tilde X^m_r,\P_{\tilde X^m_r})dW_r, & s\in[t,T],\,m\geq0
        \end{cases}
    \]\\
    Adapting to the present McKean-Vlasov framework the classical Picard's approximation method (see for instance \cite[Theorem 9.2]{Baldi}), it can be shown that
    \[
        \sup_{s\in[t,T]}|X^m_s - X^{t,\xi}_s|\ \overset{m\rightarrow +\infty}{\longrightarrow} \   0,\quad
        \sup_{s\in[t,T]}|\tilde X^m_s - X^{t,\eta}_s|\ \overset{m\rightarrow +\infty}{\longrightarrow}  \ 0\qquad \text{a.s.}
    \]
    So, in particular, for all $s\in[t,T]$
    \begin{equation}\label{weak convergence}
        \P_{X^m_s} \  \overset{d}{\longrightarrow}    \ \P_{X^{t,\xi}_s},\qquad \P_{\tilde X^m_s} \  \overset{d}{\longrightarrow}    \ \P_{X^{t,\eta}_s}\qquad\text{as } m\rightarrow+\infty,
    \end{equation}
    where $\ \overset{d}{\longrightarrow} \ $ stands for the weak convergence of distributions. By applying a similar argument to that used in the proof of Theorem 9.5 in \cite{Baldi}, we get that, for all $m\in\N$, $(X^m_s,W_s-W_t)_{s\in[t,T]}$ and $(\tilde X^m_s,W_s - W_t)_{s\in[t,T]}$ have the same law. Thus, for all $m$ and for all $s\in[t,T]$, $\P_{X^m_s}=\P_{\tilde X^{m}_s}$. This, together with \eqref{weak convergence}, directly implies that $\P_{X^{t,\xi}_s}=\P_{X^{t,\eta}_s}$, for all $s\in[t,T]$.
\end{proof}
\begin{Remark}\label{R: same marginals}
    Due to \textnormal{Proposition \ref{P: same marginal}}, given $t\in[0,T]$ and $\xi\in L^2(\Fc_t;\R^d)$ with $\P_\xi=\mu$, the marginal distributions of the corresponding state process $X^{t,\xi}$ only depend on the initial time $t$ and distribution $\mu$. Thus, for every $t\in[0,T],\, \mu\in\mathcal{P}_2(\R^d)$, we can define
    \[
        \P^{t,\mu}_s := \P_{X^{t,\xi}_s}\qquad\text{for all }s\in[t,T],
    \]
    where $\xi$ is any random variable in $L^2(\Fc_t;\R^d)$ with law $\P_\xi = \mu$.
\end{Remark}
\noindent \textbf{Reward functional and value function.} We define the reward functional $J$ and the value function $V$ for our optimal stopping problem as follows. Given $t\in[0,T],\, \xi\in L^2(\Fc_t;\R^d)$, and the corresponding state process $X^{t,\xi}$ (regarding the fact that $g$ does not depend on $\mu$, see Remark \ref{R: regularity assumptions}),
\[
    J(t,\xi,\tau) \ \coloneqq \ \E\bigg[\int_{t}^{\tau} f(s,X^{t,\xi}_s,\P_{X^{t,\xi}_s})ds \ + \ g(X^{t,\xi}_\tau)\bigg],\quad\text{for all }\tau\in\Tc_{t,T},
\]
\[
    V(t,\xi) \ \coloneqq \ \sup_{\tau\in\Tc_{t,T}} J(t,\xi,\tau).
\]
\noindent Thanks to Assumptions (\nameref{A1}) and (\nameref{A2}), both $J$ and $V$ are well-defined. %Moreover, $V$ is continuous and satisfies a sub-quadratic growth condition, as stated in the following result.
\begin{Proposition}\label{P: V properties}
    The value function $V$ is continuous and has sub-quadratic growth in $\xi$ uniformly in $t$: there exists $C>0$ such that, for all $t\in[0,T],\,\xi\in L^2(\Fc_t;\R^d)$,
    \[
        %\begin{equation}\label{growth V}
            |V(t,\xi)| \ \leq \ C(1+\|\xi\|_{L^2}^2).
        %\end{equation}
    \]
\end{Proposition}
\begin{proof}
The proof can be done proceeding as in \cite[Proposition 3.3]{cosso1}.
\end{proof}

%% file: reformulation.tex
\section{Extended mean field optimal stopping problem}\label{reformulation}
We proceed as in \cite[Section 6.5]{CD18_I} and consider the so-called enlarged state space, on which we define the extended value function \eqref{V tilde}.

\vspace{2mm}

\noindent\textbf{New state process.} Let $t\in[0,T],\, x\in\R^d,\,\mu\in\mathcal{P}_2(\R^d)$ be fixed. By Remark \ref{R: same marginals}, there exists a unique flow of probability measures $(\P^{t,\mu}_s)_{s\in[t,T]}$ associated to $(t,\mu)$. Now we take the state equation \eqref{MKV SDE} of the original problem and replace the initial state $\xi$ with $x$ and the marginal $\P_{X_s}$ with $\P^{t,\mu}_s$. So we obtain 
\begin{equation}\label{new SDE}
    \begin{cases}
        dX_s = b(s,X_s,\P^{t,\mu}_s)ds + \sigma(s,X_s,\P^{t,\mu}_s)dW_s, & s\in[t,T],\\
        X_t = x.
    \end{cases}
\end{equation}
We note that, since $(\P^{t,\mu}_s)_{s\in[t,T]}$ is fixed, \eqref{new SDE} is no longer a McKean-Vlasov type stochastic differential equation. In particular, a classical result (see for instance \cite[Theorem 9.2]{Baldi}) states that, under Assumption (\nameref{A1}), equation \eqref{new SDE} admits a unique solution in $\S^2_{t,T}$, that we denote by $X^{x,t,\mu}$.
\begin{Definition}
        Let $t\in[0,T]$, $\mu\in\mathcal{P}_2(\R^d)$ and $x\in\R^d$. The state process of the extended problem with initial data $(t,x,\mu)$ is the pair $(X^{t,x,\mu}_s,\P^{t,\mu}_s)_{s\in[t,T]}$.
\end{Definition}
\begin{Proposition}%\label{P: properties new state process}
    For every $(t,x,\mu)\in[0,T]\times\R^d\times\mathcal{P}_2(\R^d)$, and for every $p\geq 2$ the following estimate holds: there exists a constant $C_p>0$ such that
%    \begin{align}
    \begin{equation}
    \E\bigg[\sup_{t\leq s\leq T}|X^{t,x,\mu}_s|^p\bigg] \ \leq \ C_p(1 + |x|^p + \Norm{\mu}_2^p). \label{new estimate MKV 1}
    \end{equation}
\end{Proposition}
\begin{proof}
As equation \eqref{new SDE} is of non-McKean-Vlasov type, the result follows combining estimates of Proposition \ref{T:ex-uniq MKV} on $\P^{t,\mu}$ (recall from \eqref{norm_2} that $\|\P_s^{t,\mu}\|_2=\|X_s^{t,\xi}\|_{L^2}$, $s\in[t,T]$, for every $\xi\in L^2(\Fc_t;\R^d)$ with $\P_\xi=\mu$) with classical results for standard stochastic differential equations (as for instance \cite[Theorem 9.1]{Baldi}).
\end{proof}
\begin{Remark}\label{R: state processes}
    By Kolmogorov's continuity theorem it holds that, up to a suitable modification, a solution $X^{t,x,\mu}$ to \eqref{new SDE} is measurable with respect to the initial state $x\in\R^d$. Then, for every $\xi\in L^2(\Fc_t;\R^d)$ with $\P_\xi = \mu$, the process $X^{t,\xi,\mu}$ is well-defined and solves the stochastic differential equation \eqref{MKV SDE}. Hence, by \textnormal{Proposition \ref{T:ex-uniq MKV}} the two processes $X^{t,\xi,\mu}$ and $X^{t,\xi}$ are indistinguishable.
\end{Remark}

\noindent\textbf{Extended reward functional and value function.} For all $t\in[0,T]$, let $\F^{t,W}=(\Fc^{t,W}_s)_{s\in[t,T]}$ be the $\P$-completion of the filtration generated by the increments of the Brownian motion $(W_s-W_t)_{s\in[t,T]}$ and let $\tilde\Tc_{t,T}$ denote the set of $\F^{t,W}$- stopping times with values in $[t,T]$.\\
    Let $t\in[0,T]$, $x\in\R^d$, $\mu\in\mathcal{P}_2(\R^d)$ and let $(X^{t,x,\mu}_s,\P^{t,\mu}_s)_{s\in[t,T]}$ be the corresponding state process. We define
    \[
        %\begin{equation}\label{J tilde}
            \tilde J(t,x,\mu,\tau) \ \coloneqq \ \E\bigg[\int_{t}^{\tau} f(s,X^{t,x,\mu}_s,\P^{t,\mu}_s)ds \ + \ g(X^{t,x,\mu}_\tau)\bigg]\qquad\text{for all }\tau\in\tilde\Tc_{t,T}
        %\end{equation}
    \]
and the extended value function
        \begin{equation}\label{V tilde}
            \tilde V(t,x,\mu) \ \coloneqq \ \sup_{\tau\in\tilde\Tc_{t,T}} \tilde J(t,x,\mu,\tau).
        \end{equation}
Similarly to Proposition \ref{P: V properties}, we have the following result, which can be proceeding along the same lines.
\begin{Proposition}\label{P: Vtilde properties}
    The extended value function $\Tilde{V}$ is continuous and has sub-quadratic growth in $(x,\mu)$ uniformly in $t$: there exists $C>0$ such that, for all $(t,x,\mu)\in[0,T]\times\R^d\times\mathcal{P}_2(\R^d)$,
        \[
            |\tilde{V}(t,x,\mu)| \ \leq \ C(1 + |x|^2 + \Norm{\mu}_2^2).
        \]
\end{Proposition}
\begin{Remark}\label{R: why Ttilde}
    Since $\Fc^{t,W}_s\subset\Fc_s$, for every $s\in[t,T]$, it follows that $\tilde\Tc_{t,T}\subset\Tc_{t,T}$. The choice of $\tilde\Tc_{t,T}$ in the definition \eqref{V tilde} of $\tilde V$ plays a key role in the proof of \textnormal{Theorem \ref{T:representation V tilde}}, where we initially require stopping times to be independent of the initial $\sigma$-algebra. We will prove in \textnormal{Corollary \ref{C: enlarge stopping times for Vtilde}} that we still obtain the same $\tilde V$ if we take the supremum over $\Tc_{t,T}$ in \eqref{V tilde}. Notice however that the choice of $\tilde\Tc_{t,T}$ in \eqref{V tilde} is consistent with the reformulation of our problem. Indeed, while in the original problem the evolution of the state process $X^{t,\xi}$ depends on two independent random components $($the $\Fc_t$-measurable initial state $\xi$ and the increments of the Brownian motion $(W_s-W_t)_{s\in[t,T]}$$)$, in the extended problem the only source of noise for the state process $X^{t,x,\mu}$ is provided by the Brownian increments.
\end{Remark}
\begin{Remark}
    Let  $t\in[0,T]$, $\mu\in\mathcal{P}_2(\R^d)$, $\tau\in\tilde\Tc_{T,t}$, and $\xi\in L^2(\Fc_t;\R^d)$ with $\P_\xi = \mu$. In general $\tilde J(t,\xi,\mu,\tau)\neq J(t,\xi,\tau)$, since $\tilde J(t,\xi,\mu,\tau)$ is a random variable (given by the composition of $\tilde J(t,x,\mu,\tau)$ and $\xi$), while $J(t,\xi,\tau)$ is by definition a deterministic function. Similarly, in general $\tilde V(t,\xi,\mu)\neq V(t,\xi)$. For the precise relation between $\tilde V$ and $V$ see Corollary \ref{C:relation V - V tilde}.
\end{Remark}

\noindent We now proceed to prove the following key equality for the value function $\tilde{V}$. This result is of critical importance for the entire work, as it leads to two fundamental consequences: the explicit formula for the optimal stopping time of the original problem and the dynamic programming principle for the extended problem.
\begin{Theorem}\label{T:representation V tilde}
    Let $t\in[0,T]$, $\eta\in L^2(\Fc_t;\R^d)$ and $\mu \in\mathcal{P}_2(\R^d)$. Then
    \[
    %\begin{equation}\label{rep Vtilde}
        \tilde V(t,\eta,\mu) \ = \ \esssup_{\tau\in\Tc_{t,T}} \E\bigg[\int_{t}^{\tau} f(s,X^{t,\eta,\mu}_s,\P^{t,\mu}_s)ds \ + \ g(X^{t,\eta,\mu}_\tau) \bigg|\Fc_t\bigg], \qquad \P\text{-a.s.}
    %\end{equation}
    \]
\end{Theorem}
\noindent In order to prove Theorem \ref{T:representation V tilde} we need the following preliminary result.
\begin{Lemma}\label{L: stopping times}
    Let $t\in[0,T]$ and $\tau\in\Tc_{t,T}$. Then there exists a sequence $(K_n)_n\subset\N$ and a sequence of stopping times $(\tau_n)_n\in\Tc_{t,T}$ converging to $\tau$ almost surely and such that, for every $n$,
    \[
    %\begin{equation}\label{successione tau}
        \tau_n = \sum_{i=1}^{K_n}\tau_{n,i}\mathbbm{1}_{B_{n,i}}, 
    %\end{equation}
    \]
     where $\tau_{n,i}\in\tilde\Tc_{t,T}$ for all $i\in\{1,...,K_n\}$ and $\{B_{n,1},...,B_{n,K_n}\}\subset\Fc_t$ is a partition of $\Omega$.
\end{Lemma}
\begin{proof}
See \cite[Theorem 3.1]{CossoPerelli_Approx}.
\end{proof}
\begin{proof}[Proof of Theorem \ref{T:representation V tilde}]
    We split the proof into several steps.
   
   \vspace{2mm}
   
   \noindent\textsc{Step} 1. We begin by proving the inequality
    \begin{equation}\label{ClaimStep1}
        \tilde V(t,\eta,\mu) \ \geq \ \esssup_{\tau\in\Tc_{t,T}} \E\bigg[\int_{t}^{\tau} f(s,X^{t,\eta,\mu}_s,\P^{t,\mu}_s)ds \ + \ g(X^{t,\eta,\mu}_\tau) \bigg|\Fc_t\bigg]
    \end{equation}
    through the following three steps.

    \vspace{2mm}
    
    \noindent\textsc{Step} 1.1. We suppose that $\eta = \sum_{i=1}^n x_i\mathbbm{1}_{A_i}$ where $x_i\in\R^d$ and $(A_i)_i\subset\Fc_t$ is a partition of $\Omega$. Since $\tilde V(t,\eta,\mu)$ is given by the composition of $\tilde V(t,x,\mu)$ with $\eta$, we have
    \begin{equation*}
        \tilde V(t,\eta,\mu) \ = \ \sum_{i=1}^n     \tilde V(t,x_i,\mu)\,\mathbbm{1}_{A_i}.
    \end{equation*}
    By definition of $\tilde V(t,x_i,\mu)$, for every $\tau\in\tilde\Tc_{t,T}$ fixed, it holds that
    \[
        \tilde V(t,\eta,\mu) \ \geq \ \sum_{i=1}^n \E\bigg[\int_t^{\tau} f(s, X^{t,x_i,\mu}_s, \P^{t,\mu}_s)ds + g(X^{t,x_i,\mu}_{\tau})\bigg]\mathbbm{1}_{A_i}.
    \] 
    By definition of $\tilde\Tc_{t,T}$, $\tau$ is independent of $\Fc_t$. Similarly, $X^{t,x_i,\mu}$ is independent of $\Fc_t$ as it depends only on the increments $(W_s-W_t)_{s\in[t,T]}$. Then, we can replace the expected value in the latter inequality with a conditional expectation as follows:
    
    \[
        \tilde V(t,\eta,\mu) \ \geq \ \sum_{i=1}^n \E\bigg[\int_t^{\tau} f(s, X^{t,x_i,\mu}_s, \P^{t,\mu}_s)ds + g(X^{t,x_i,\mu}_{\tau})\bigg|\Fc_t\bigg]\mathbbm{1}_{A_i}.
    \]
    Now, since $A_i\in\Fc_t$ for every $i$,
    \begin{align*}
        \tilde V(t,\eta,\mu) \ &\geq \ \sum_{i=1}^n \E\bigg[\mathbbm{1}_{A_i}\bigg(\int_t^{\tau} f(s, X^{t,x_i,\mu}_s, \P^{t,\mu}_s)ds + g(X^{t,x_i,\mu}_{\tau})\bigg)\bigg|\Fc_t\bigg]\\
        &= \ \E\bigg[\sum_{i=1}^n \mathbbm{1}_{A_i}\bigg(\int_t^{\tau} f(s, X^{t,x_i,\mu}_s, \P^{t,\mu}_s)ds + g(X^{t,x_i,\mu}_{\tau})\bigg)\bigg|\Fc_t\bigg]\\
        &= \ \E\bigg[\sum_{i=1}^n \mathbbm{1}_{A_i}\bigg(\int_t^{\tau} f(s, X^{t,\eta,\mu}_s, \P^{t,\mu}_s)ds + g(X^{t,\eta,\mu}_{\tau})\bigg)\bigg|\Fc_t\bigg],
    \end{align*}
    where in the last inequality we used that $\eta = x_i$ on each set $A_i$. As a consequence, recalling that $(A_i)_i$ is a partition of $\Omega$,
    \begin{equation}
        \tilde V(t,\eta,\mu) \ \geq \ \E\bigg[\int_t^{\tau} f(s, X^{t,\eta,\mu}_s, \P^{t,\mu}_s)ds + g(X^{t,\eta,\mu}_{\tau})\bigg|\Fc_t\bigg]. \label{passo 1.1}
    \end{equation}
   Furthermore, by linearity, the inequality \eqref{passo 1.1} is verified also for stopping times of the form $\tau = \sum_{i=1}^n\tau_i\mathbbm{1}_{A_i}$, with $\tau_i\in\tilde\Tc_{t,T}$ and $A_i$ as above.
   
   \vspace{2mm}
   
   \noindent\textsc{Step} 1.2. We prove \eqref{passo 1.1} with $\eta$ as in \textsc{Step} 1.1 but with $\tau\in\Tc_{t,T}$. By Lemma \ref{L: stopping times} there exists a sequence of stopping times $(\tau_n)_n\subset\Tc_{t,T}$ converging almost surely to $\tau$ and such that
   \[
        \tau_n = \sum_{i=1}^{K_n}\tau_{n,i}\mathbbm{1}_{A_{n,i}},
   \] 
   where $\tau_{n,i}\in\tilde\Tc_{t,T}$ and $\{A_{n,1},...,A_{n,K_n}\}\subset\Fc_t$ is a partition of $\Omega$ for each $n$. By the previous step, it holds that
    \[
        \tilde V(t,\eta,\mu) \ \geq \  \E\bigg[\int_t^{\tau_n} f(s, X^{t,\eta,\mu}_s, \P^{t,\mu}_s)ds + g(X^{t,\eta,\mu}_{\tau_n})\bigg|\Fc_t\bigg].
    \]
    Then, we only need to prove that
    \begin{align}
        \E\bigg[\int_t^{\tau_n} f(s, &X^{t,\eta,\mu}_s, \P^{t,\mu}_s)ds + g(X^{t,\eta,\mu}_{\tau_n})\bigg|\Fc_t\bigg] \label{Proof_Step2} \\
        &\overset{n\to+\infty}{\longrightarrow} \ \E\bigg[\int_t^{\tau} f(s, X^{t,\eta,\mu}_s, \P^{t,\mu}_s)ds + g(X^{t,\eta,\mu}_{\tau} )\bigg|\Fc_t\bigg]\quad\text{a.s.} \notag
    \end{align}
    By the continuity of $g$ and the fact that $X^{t,\eta,\mu}$ has continuous trajectories, we get
    \[
        g(X^{t,\eta,\mu}_{\tau_n}) \ \longrightarrow \ g(X^{t,\eta,\mu}_{\tau})\qquad\text{a.s.}\qquad\text{as }n\to+\infty.
    \]
    On the other hand, by Lebesgue's dominated convergence theorem,  we have that
    \begin{align*}
        &\int_t^{\tau_n} f(s, X^{t,\eta,\mu}_s, \P^{t,\mu}_s)ds \ = \ \int_t^{T} f(s, X^{t,\eta,\mu}_s, \P^{t,\mu}_s)\mathbbm{1}_{[t,\tau_n]}(s)ds\\
        & \overset{n\to+\infty}{\longrightarrow} \int_t^{T} f(s, X^{t,\eta,\mu}_s, \P^{t,\mu}_s)\mathbbm{1}_{[t,\tau]}(s)ds \ = \ \int_t^{\tau} f(s, X^{t,\eta,\mu}_s, \P^{t,\mu}_s)ds\quad\text{a.s.}
    \end{align*}
    We can use Lebesgue's dominated convergence theorem thanks to the quadratic growth of $f$ in $(x,\mu)$ and the fact that $X^{t,\eta,\mu}$ and $\P^{t,\mu}$ have continuous trajectories; more precisely, we have
    \begin{align*}
        &\big|f(s, X^{t,\eta,\mu}_s, \P^{t,\mu}_s)\mathbbm{1}_{[t,\tau_n]}(s)-f(s, X^{t,\eta,\mu}_s, \P^{t,\mu}_s)\mathbbm{1}_{[t,\tau]}(s)\big| \ \leq \ 2|f(s, X^{t,\eta,\mu}_s, \P^{t,\mu}_s)| \\
        &\leq \ 2C(1+|X^{t,\eta,\mu}_s|^2+\Norm{ \P^{t,\mu}_s}_2^2) \ \leq \ 2C\bigg(1+\sup_{s\in[t,T]}|X^{t,\eta,\mu}_s|^2+\sup_{s\in[t,T]}\Norm{ \P^{t,\mu}_s}_2^2\bigg)\eqqcolon M \ < \ +\infty.
    \end{align*}
    Besides, it holds that
    \begin{align*}
        &\bigg|\int_t^{T} f(s, X^{t,\eta,\mu}_s, \P^{t,\mu}_s)\mathbbm{1}_{[t,\tau_n]}(s)ds-\int_t^{T} f(s, X^{t,\eta,\mu}_s, \P^{t,\mu}_s)\mathbbm{1}_{[t,\tau]}(s)ds\bigg|+|g(X^{t,\eta,\mu}_{\tau_n})- g(X^{t,\eta,\mu}_{\tau})|\\
        &\leq \ \int_t^{T}|f(s, X^{t,\eta,\mu}_s, \P^{t,\mu}_s)||\mathbbm{1}_{[t,\tau_n]}(s)-\mathbbm{1}_{[t,\tau]}(s)|ds + |g(X^{t,\eta,\mu}_{\tau_n})|+|g(X^{t,\eta,\mu}_{\tau})|\\
        &\leq \ 2C\int_t^{T}\bigg(1+\sup_{r\in[t,T]}|X^{t,\eta,\mu}_r|^2+\sup_{r\in[t,T]}\Norm{ \P^{t,\mu}_r}_2^2\bigg)ds + 2C\bigg(1+\sup_{r\in[t,T]}|X^{t,\eta,\mu}_r|^2\bigg)\\
        &\leq \ 2CT\bigg(1+\sup_{r\in[t,T]}|X^{t,\eta,\mu}_r|^2+\sup_{r\in[t,T]}\Norm{ \P^{t,\mu}_r}_2^2\bigg)+2C\bigg(1+\sup_{r\in[t,T]}|X^{t,\eta,\mu}_r|^2\bigg) \ \in \ L^1(\Omega).
    \end{align*}
    Thus, by the conditional dominated convergence theorem, for every $\tau\in\Tc_{t,T}$ and $\eta = \sum_{i=1}^n x_i\mathbbm{1}_{A_i}$ we have that \eqref{passo 1.1} holds.

    \vspace{2mm}
    
    \noindent\textsc{Step} 1.3. In this final step we consider a generic $\eta\in L^2(\Fc_t;\R^d)$. It is well-known that $\eta$ can be approximated in $L^2$-norm and in the $\P$-almost sure sense by a sequence of simple random variables $(\eta_n)_n\subset L^2(\Fc_t;\R^d)$ of the form
    \[
        \eta_n = \sum_{i=1}^{N_n} x_{i,n}\mathbbm{1}_{A_{n,i}}
    \]
    where $x_{i,n}\in\R^d$ and $\{A_{n,1},...,A_{n,N_n}\}\subset\Fc_t$ is a partition of $\Omega$ for all $n$. 
    % In particular, we can assume that this sequence converges almost everywhere to $\eta$.
    By the previous step, for every $n$ and for every $\tau\in\Tc_{t,T}$, it holds that
    \begin{equation}\label{passo 1.3}
        \tilde V(t,\eta_n,\mu) \ \geq \   \E\bigg[\int_t^{\tau} f(s, X^{t,\eta_n,\mu}_s, \P^{t,\mu}_s)ds + g(X^{t,\eta_n,\mu}_{\tau})\bigg|\Fc_t\bigg].
    \end{equation}
    To conclude the proof, it remains to take the limit as $n$ tends to infinity. From Proposition \ref{P: Vtilde properties}, we already know that, since $\eta_n\longrightarrow\eta$ pointwise almost surely,
    \[
        \tilde{V}(t,\eta_n,\mu) \ \longrightarrow \ \tilde{V}(t,\eta,\mu)\qquad\text{a.s.}\qquad\text{as }n\rightarrow+\infty.
    \]
    So we only need to show that, for every $\tau\in\Tc_{t,T}$,
    \begin{align}
        \E\bigg[\int_t^{\tau} f(s, &X^{t,\eta_n,\mu}_s, \P^{t,\mu}_s)ds + g(X^{t,\eta_n,\mu}_{\tau})\bigg|\Fc_t\bigg] \label{ProofStep1.3} \\
        &\overset{n\to+\infty}{\longrightarrow} \ \E\bigg[\int_t^{\tau} f(s, X^{t,\eta,\mu}_s, \P^{t,\mu}_s)ds + g(X^{t,\eta,\mu}_{\tau})\bigg|\Fc_t\bigg]\quad\text{a.s.} \notag
    \end{align}
    The proof can be done proceeding along the same lines as in the proof of \eqref{Proof_Step2}. As a consequence, letting $n\to+\infty$ in \eqref{passo 1.3}, we get
    \[
        \tilde V(t,\eta,\mu) \ \geq \ \E\bigg[\int_t^{\tau} f(s, X^{t,\eta,\mu}_s, \P^{t,\mu}_s)ds + g(X^{t,\eta,\mu}_{\tau} )\bigg|\Fc_t\bigg],\qquad\text{for all }\tau\in\Tc_{t,T}.
    \]
    This implies, by definition of essential supremum, that \eqref{ClaimStep1} holds.

    \vspace{2mm}
    
    \noindent\textsc{Step} 2. We prove the opposite inequality
    \begin{equation}\label{ClaimStep2}
        \tilde V(t,\eta,\mu) \ \leq \ \esssup_{\tau\in\Tc_{t,T}} \E\bigg[\int_{t}^{\tau} f(s,X^{t,\eta,\mu}_s,\P^{t,\mu}_s)ds \ + \ g(X^{t,\eta,\mu}_\tau
        ) \bigg|\Fc_t\bigg].
    \end{equation}
    \textsc{Step} 2.1. We suppose that $\eta = \sum_{i=1}^n x_i\mathbbm{1}_{A_i}$ where $x_i\in\R^d$ and $(A_i)_i\subset\Fc_t$ is a partition of $\Omega$. By definition of $\tilde V(t,x_i,\mu)$, for every $\varepsilon>0$ and  $i$ fixed, there exists $\tau_{\varepsilon,i}\in\tilde\Tc_{t,T}$ such that
    \begin{align*}
        \tilde V(t,\eta,\mu) \ &= \ \sum_{i=1}^n \tilde V(t,x_i,\mu)\mathbbm{1}_{A_i}\\
        &\leq \ \sum_{i=1}^n \bigg(\E\bigg[\int_t^{\tau_{\varepsilon,i}} f(s, X^{t,x_i,\mu}_s, \P^{t,\mu}_s)ds + g(X^{t,x_i,\mu}_{\tau_{\varepsilon,i}})\bigg] \ + \ \varepsilon\bigg)\mathbbm{1}_{A_i}\\
        &= \ \sum_{i=1}^n \E\bigg[\int_t^{\tau_{\varepsilon,i}} f(s, X^{t,x_i,\mu}_s, \P^{t,\mu}_s)ds + g(X^{t,x_i,\mu}_{\tau_{\varepsilon,i}})\bigg]\mathbbm{1}_{A_i} \ + \ \varepsilon.
    \end{align*}
    Proceeding as in \textsc{Step} 1.1, we have
    \begin{align*}
        \tilde V(t,\eta,\mu) \ &\leq \ \sum_{i=1}^n \E\bigg[\int_t^{\tau_{\varepsilon,i}} f(s, X^{t,x_i,\mu}_s, \P^{t,\mu}_s)ds + g(X^{t,x_i,\mu}_{\tau_{\varepsilon,i}})\bigg|\Fc_t\bigg]\mathbbm{1}_{A_i} \ + \ \varepsilon\\
        &= \ \E\bigg[\sum_{i=1}^n \mathbbm{1}_{A_i}\bigg(\int_t^{\tau_{\varepsilon,i}} f(s, X^{t,x_i,\mu}_s, \P^{t,\mu}_s)ds + g(X^{t,x_i,\mu}_{\tau_{\varepsilon,i}})\bigg)\bigg|\Fc_t\bigg]\ + \ \varepsilon.
    \end{align*}
    Let $\tau_\varepsilon \coloneqq \sum_{i=1}^n \tau_{\varepsilon,i}\mathbbm{1}_{A_i}$. It holds that $\tau_\varepsilon\in\Tc_{t,T}$, and then
    \begin{align*}
        \tilde V(t,\eta,\mu) \ &\leq \ \E\bigg[\sum_{i=1}^n \mathbbm{1}_{A_i}\bigg(\int_t^{\tau_{\varepsilon}} f(s, X^{t,\eta,\mu}_s, \P^{t,\mu}_s)ds + g(X^{t,\eta,\mu}_{\tau_{\varepsilon}})\bigg)\bigg|\Fc_t\bigg]\ + \ \varepsilon\\
        &= \ \E\bigg[\int_t^{\tau_{\varepsilon}} f(s, X^{t,\eta,\mu}_s, \P^{t,\mu}_s)ds + g(X^{t,\eta,\mu}_{\tau_{\varepsilon}})\bigg|\Fc_t\bigg]\ + \ \varepsilon.
    \end{align*}
    By the definition of essential supremum and from the arbitrariness of $\varepsilon$, we see that \eqref{ClaimStep2} holds.

    \vspace{2mm}
    
    \noindent\textsc{Step} 2.2. Consider a generic $\eta\in L^2(\Fc_t;\R^d)$. As in \textsc{Step} 1.2, we can find a sequence of simple random variables $(\eta_n)_n\subset L^2(\Fc_t;\R^d)$ which converges in $L^2$ and in the $\P$-almost sure sense to $\eta$, so that
    \[
        \tilde V(t,\eta_n,\mu) \ \longrightarrow \ \tilde V(t,\eta,\mu)\qquad\text{a.s.}\qquad\text{as }n\to+\infty.
    \]
    To conclude the proof it is enough to show that
    \begin{align*}
        \esssup_{\tau\in\Tc_{t,T}}\E\bigg[\int_t^{\tau} f(s, &X^{t,\eta_n,\mu}_s, \P^{t,\mu}_s)ds + g(X^{t,\eta_n,\mu}_{\tau})\bigg|\Fc_t\bigg]\\
        &\overset{n\to+\infty}{\longrightarrow}\esssup_{\tau\in\Tc_{t,T}}\E\bigg[\int_t^{\tau} f(s, X^{t,\eta,\mu}_s, \P^{t,\mu}_s)ds + g(X^{t,\eta,\mu}_{\tau})\bigg|\Fc_t\bigg]\quad\text{a.s.}
    \end{align*}
    Since for every collection of real-valued random variables $\Hc$ on $(\Omega,\Fc,\P)$ it holds that $(\esssup_{\Hc}X$ $-$ $\esssup_{\Hc}Y)$ $\leq$ $\esssup_{\Hc}(X-Y)$ a.s., we get
    \begin{align*}
        &\bigg|\esssup_{\tau\in\Tc_{t,T}} \E\bigg[\int_t^{\tau} \big(f(s, X^{t,\eta_n,\mu}_s, \P^{t,\mu}_s) - f(s, X^{t,\eta,\mu}_s, \P^{t,\mu}_s)\big)ds + \big(g(X^{t,\eta_n,\mu}_{\tau}) - g(X^{t,\eta,\mu}_{\tau})\big) \bigg|\Fc_t\bigg]\bigg|\\
        &\leq \ \esssup_{\tau\in\Tc_{t,T}}\E\bigg[\int_t^{\tau} \big|f(s,X^{t,\eta_n,\mu}_s, \P^{t,\mu}_s) - f(s, X^{t,\eta,\mu}_s, \P^{t,\mu}_s)\big|ds + \big|g(X^{t,\eta_n,\mu}_{\tau}) - g(X^{t,\eta,\mu}_{\tau})\big| \bigg|\Fc_t\bigg].
    \end{align*}
    As seen in \textsc{Step} 1.3, for each $\tau\in\Tc_{t,T}$ we have that \eqref{ProofStep1.3} holds. In particular, the convergence is uniform in $\tau\in\Tc_{t,T}$, that is
    \[
        \esssup_{\tau\in\Tc_{t,T}}\E\bigg[\int_t^{\tau} \big|f(s,X^{t,\eta_n,\mu}_s, \P^{t,\mu}_s) - f(s, X^{t,\eta,\mu}_s, \P^{t,\mu}_s)\big|ds \ + \ \big|g(X^{t,\eta_n,\mu}_{\tau}) - g(X^{t,\eta,\mu}_{\tau})\big| \bigg|\Fc_t\bigg]\longrightarrow 0
    \]
    almost surely as $n$ goes to infinity, from which the claim follows.
\end{proof}
\noindent As a consequence of Theorem \ref{T:representation V tilde}, we get that the extended value function $\Tilde{V}$, which we have originally defined as the supremum over $\Tilde{\Tc}_{t,T}$, can be equivalently defined as the supremum over the larger set $\Tc_{t,T}$.
\begin{Corollary}\label{C: enlarge stopping times for Vtilde}
    Let $t\in[0,T],\,x\in\R^d,\,\mu\in\mathcal{P}_2(\R^d)$. Then 
    \[
        \Tilde{V}(t,x,\mu) \ = \ \sup_{\tau\in\Tc_{t,T}}\E\bigg[\int_t^{\tau} f(s, X^{t,x,\mu}_s, \P^{t,\mu}_s)ds + g(X^{t,x,\mu}_{\tau})\bigg].
    \]
\end{Corollary}
\begin{proof}
    By definition of $\Tilde{V}$ and by the inclusion $\Tilde{\Tc}_{t,T}\subset\Tc_{t,T}$, we immediately obtain the inequality
    \[
        \tilde V(t,x,\mu) \ \leq \ \sup_{\tau\in\Tc_{t,T}}\E\bigg[\int_t^{\tau} f(s, X^{t,x,\mu}_s, \P^{t,\mu}_s)ds + g(X^{t,x,\mu}_{\tau})\bigg].
    \]
    It remains to prove the opposite inequality. Using Theorem \ref{T:representation V tilde} with $\eta\equiv x$, we have 
    \[
        \tilde V(t,x,\mu) \ = \ \esssup_{\tau\in\Tc_{t,T}}\E\bigg[\int_t^{\tau} f(s, X^{t,x,\mu}_s, \P^{t,\mu}_s)ds + g(X^{t,x,\mu}_{\tau})\bigg|\Fc_t\bigg].
    \]
    As a consequence, for each $\tau\in\Tc_{t,T}$, by definition of essential supremum,
    \[
        \tilde V(t,x,\mu) \ \geq \ \E\bigg[\int_t^{\tau} f(s, X^{t,x,\mu}_s, \P^{t,\mu}_s)ds + g(X^{t,x,\mu}_{\tau})\bigg|\Fc_t\bigg].
    \]
    Since the function $\Tilde{V}$ is deterministic, we get
    \[
        \tilde V(t,x,\mu) \ = \ \E[\Tilde{V}(t,x,\mu)] \ \geq \  \E\bigg[\int_t^{\tau} f(s, X^{t,x,\mu}_s, \P^{t,\mu}_s)ds + g(X^{t,x,\mu}_{\tau})\bigg].
    \]
    Hence, by the arbitrariness of $\tau$,
    \[
        \Tilde{V}(t,x,\mu) \ \geq \ \sup_{\tau\in\Tc_{t,T}}\E\bigg[\int_t^{\tau} f(s, X^{t,x,\mu}_s, \P^{t,\mu}_s)ds + g(X^{t,x,\mu}_{\tau})\bigg],
    \]
    and this concludes the proof.
\end{proof}
\noindent Another consequence of Theorem \ref{T:representation V tilde} is the following ``disintegration'' result, which shows the link between the original value function $V$ and the extended value function $\tilde V$.
\begin{Corollary}\label{C:relation V - V tilde}
    Let $t\in[0,T]$, $\xi\in L^2(\Fc_t;\R^d)$, and $\mu=\P_\xi\in\mathcal{P}_2(\R^d)$. Then
    \[
        V(t,\xi) \ = \ \E\big[\tilde V(t,\xi,\mu)\big] \ = \ \int_{\R^d}\tilde V(t,x,\mu)\mu(dx).
    \]
\end{Corollary}
\begin{proof}
    We split the proof into two steps.

    \vspace{2mm}
    
    \noindent\textsc{Step} 1. \emph{Proof of the inequality} $V(t,\xi)$ $\leq$ $\E\big[\tilde V(t,\xi,\mu)\big]$. Using Theorem \ref{T:representation V tilde} with $\eta = \xi$, we have that, for any $\tau\in\Tc_{t,T}$,
    \[
        \tilde V(t,\xi,\mu) \ \geq \ \E\bigg[\int_{t}^{\tau} f(s,X^{t,\xi,\mu}_s,\P^{t,\mu}_s)ds \ + \ g(X^{t,\xi,\mu}_\tau) \bigg|\Fc_t\bigg].
    \]
    Thus,
    \[
        \E\big[\tilde V(t,\xi,\mu)\big] \ \geq \ \E\bigg[\int_{t}^{\tau} f(s,X^{t,\xi,\mu}_s,\P^{t,\mu}_s)ds \ + \ g(X^{t,\xi,\mu}_\tau)\bigg].
    \]
    Since $X^{t,\xi,\mu}$ and $X^{t,\xi}$ are indistinguishable processes (see Remark \ref{R: state processes}), it holds that
    \[
        \E\big[\tilde V(t,\xi,\mu)\big] \ \geq \ \E\bigg[\int_{t}^{\tau} f(s,X^{t,\xi}_s,\P^{t,\mu}_s)ds \ + \ g(X^{t,\xi}_\tau)\bigg] \ = \  J(t,\xi,\tau).
    \]
    Hence, by the arbitrariness of $\tau\in\Tc_{t,T}$, 
    \[
        \E\big[\tilde V(t,\xi,\mu)\big] \ \geq \ V(t,\xi).
    \]
    
    \noindent\textsc{Step} 2. \emph{Proof of the inequality} $V(t,\xi)$ $\geq$ $\E\big[\tilde V(t,\xi,\mu)\big]$. Using Theorem \ref{T:representation V tilde} with $\eta=\xi$ again, for any $\varepsilon>0$ there exists
    $\tau_\varepsilon\in\Tc_{t,T}$ such that
    \[
        \tilde V(t,\xi,\mu) \ \leq \ \E\bigg[\int_{t}^{\tau_\varepsilon} f(s,X^{t,\xi,\mu}_s,\P^{t,\mu}_s)ds \ + \ g(X^{t,\xi,\mu}_{\tau_\varepsilon}) \bigg|\Fc_t\bigg] \ + \ \varepsilon.
    \]
    This implies
     \[
        \E\big[\tilde V(t,\xi,\mu)\big] \ \leq \ \E\bigg[\int_{t}^{\tau_\varepsilon} f(s,X^{t,\xi,\mu}_s,\P^{t,\mu}_s)ds \ + \ g(X^{t,\xi,\mu}_{\tau_\varepsilon}) \bigg] \ + \ \varepsilon.
    \]
    Following the same reasoning as in \textsc{Step} 1, we get
    \[
        \E\big[\tilde V(t,\xi,\mu)\big] \ \leq \ J(t,\xi,\tau_\varepsilon) \ + \ \varepsilon.
    \]
    Hence
    \[
        \E\big[\tilde V(t,\xi,\mu)\big] \ \leq \ \sup_{\tau\in\Tc_{t,T}}J(t,\xi,\tau) \ + \ \varepsilon
    \]
    By definition of $V$ and the arbitrariness of $\varepsilon$, we conclude that $\E\big[\tilde V(t,\xi,\mu)\big]$ $\leq$ $V(t,\xi)$.
\end{proof}

%% file: snell_envelope.tex
\section{Snell envelope and optimal stopping time}\label{S:Snell}
Let $t\in[0,T],\,\xi\in L^2(\Fc_t;\R^d)$ be fixed and define as in \cite[Section 2, Chapter I]{PS} the gain process
\[
%\begin{equation}\label{process G}
    G^{t,\xi}_s \ \coloneqq \ \int_{t}^{s}f\big(r,X^{t,\xi}_r,\P_{X^{t,\xi}_r}\big)dr \ + \ g\big(X^{t,\xi}_s\big),\qquad\text{for all }s\in[t,T].
%\end{equation}
\]
Then, by definition, we can rewrite the original value function $V$ as
\[
    V(t,\xi) \ = \  \sup_{\tau\in\Tc_{t,T}}\E\big[G^{t,\xi}_\tau\big].
\]
Notice that $\E[\sup_{t\leq s\leq T}|G^{t,\xi}_s|]<+\infty$, so that it is well-defined the process
\begin{equation}\label{snell envelope}
    \Sc^{t,\xi}_s \ \coloneqq \ \esssup_{\tau\in\Tc_{s,T}}\E\big[G^{t,\xi}_\tau\big|\Fc_s\big],\qquad\text{for all }s\in[t,T].
\end{equation}
According to \cite[Theorem 2.2]{PS}, $\Sc^{t,\xi}$ is the Snell envelope of $G^{t,\xi}$, while the smallest optimal stopping time $\hat{\tau}^{t,\xi}\in\Tc_{t,T}$ for the problem with initial data $t$ and $\xi$ is given by
\begin{equation}\label{optimal tau}
    \hat{\tau}^{t,\xi} \ \coloneqq \ \inf\big\{s\in[t,T]\ \big| \ \Sc^{t,\xi}_s = G^{t,\xi}_s\big\}.
\end{equation}
The Snell envelope and the optimal stopping time can be written in terms of the extended value function $\tilde V$ in \eqref{V tilde}, as stated in the following theorem.
\begin{Theorem}\label{T:Snell}
    Let $t\in[0,T],\,\xi\in L^2(\Fc_t;\R^d)$ and let $X^{t,\xi}$ be the corresponding state process. Then the process
    \[
        \bigg(\Tilde{V}(s,X^{t,\xi}_s, \P_{X^{t,\xi}_s}) \ + \ \int_{t}^{s} f(r,X^{t,\xi}_r,\P_{X^{t,\xi}_r})dr\bigg)_{s\in[t,T]}
    \]
    coincides with the Snell envelope $\Sc^{t,\xi}$. Moreover, the optimal stopping time $\hat{\tau}^{t,\xi}$ can be written as
    \[
        \hat{\tau}^{t,\xi} \ = \ \inf\big\{s\in[t,T] \ \big| \ \Tilde{V}(s,X^{t,\xi}_s,\P_{X^{t,\xi}_s}) = g(X^{t,\xi}_s)\big\}.
    \]
\end{Theorem}
\noindent In order to prove Theorem \ref{T:Snell}, we need the so-called \emph{flow property} of the state process (see also Remark \ref{R: state processes}).
\begin{Lemma}[Flow property]\label{L: flow prop}
    Let $t\in[0,T]$, $x\in\R^d$, $\mu\in\mathscr{P}_2(\R^d)$ and let $(X^{t,x,\mu}_s,\P^{t,\mu}_s)_{s\in[t,T]}$ be the corresponding state process. Then, for every $t'\in[t,T]$,
    \begin{equation}\label{eq prop flusso}
        \big(X^{t,x,\mu}_s,\P^{t,\mu}_s\big)_{s\in[t',T]} \ = \ \big(X^{t',X^{t,x,\mu}_{t'}, \P^{t,\mu}_{t'}}_s,\P^{t',\P^{t,\mu}_{t'}}_s\big)_{s\in[t',T]}\quad\text{a.s.}
    \end{equation}
\end{Lemma}
\begin{proof}We split the proof into two steps.

    \vspace{2mm}
    
    \noindent\textsc{Step} 1. \emph{Proof of the equality} $\P^{t,\mu}_s \ = \ \P^{t',\P^{t,\mu}_{t'}}_s$ for all $s\in[t',T]$. Let $\xi\in L^2(\Fc_t;\R^d)$ be such that $\P_\xi = \mu$ and let $X^{t,\xi}$ be the corresponding solution to \eqref{MKV SDE}. By Proposition \ref{P: same marginal} and Remark \ref{R: same marginals}, we know that the flow of marginals of $X^{t,\xi}$ is $(\P^{t,\mu}_s)_{s\in[t',T]}$. Let $s\in[t',T]$, then
    \begin{align*}
        X^{t,\xi}_s \ &= \ \xi + \int_t^s b(r,X^{t,\xi}_r,\P^{t,\mu}_r)dr + \int_t^s \sigma(r,X^{t,\xi}_r,\P^{t,\mu}_r)dW_r\\
        &= \ X^{t,\xi}_{t'} + \int_{t'}^{s} b(r,X^{t,\xi}_r,\P^{t,\mu}_r)dr + \int_{t'}^s \sigma(r,X^{t,\xi}_r,\P^{t,\mu}_r)dW_r.
    \end{align*}
    Thus, $(X_s^{t,\xi})_{s\in[t',T]}$ solves the stochastic differential equation whose unique solution (under Assumption (\nameref{A1})) is $X^{t',X^{t,\xi}_{t'}}$; it follows that the processes $X^{t,\xi}$ e $X^{t',X^{t,\xi}_{t'}}$ are indistinguishable on $[t',T]$. Then, by Proposition \ref{P: same marginal} and by the fact that $\P_{X^{t,\xi}_{t'}} = \P^{t,\mu}_{t'}$, it holds that
    \[
        \P^{t,\mu}_s \ = \ \P_{X^{t,\xi}_s} \ = \ \P_{X^{t',X^{t,\xi}_{t'}}_s} \ = \ \P^{t',\P^{t,\mu}_{t'}}_s,\qquad\text{for all} s\in[t',T].
    \]
    \textsc{Step} 2. \emph{Proof that $X^{t,x,\mu}$ and $X^{t',X^{t,x,\mu}_{t'},\P^{t,\mu}_{t'}}$ are indistinguishable processes on $[t',T]$}. Let $s\in[t',T]$. By definition, the process $X^{t,x,\mu}$ satisfies
    \begin{align*}
        X^{t,x,\mu}_s \ &= \ x \ + \ \int_t^s b(r,X^{t,x,\mu}_r, \P^{t,\mu}_r)dr \ + \ \int_t^s \sigma(r,X^{t,x,\mu}_r, \P^{t,\mu}_r)dW_r\\
        &= \ X^{t,x,\mu}_{t'} \ + \ \int_{t'}^s b(r,X^{t,x,\mu}_r, \P^{t,\mu}_r)dr \ + \ \int_{t'}^s \sigma(r,X^{t,x,\mu}_r,\P^{t,\mu}_r)dW_r.
    \end{align*}
    By \textsc{Step} 1 we have $\P^{t,\mu}_r = \P^{t',\P^{t,\mu}_{t'}}_r$, for every $r\in[t',T]$. Then, we can write
    \[
        X^{t,x,\mu}_s \ = \ \ X^{t,x,\mu}_{t'} \ + \ \int_{t'}^s b(r,X^{t,x,\mu}_r, \P^{t',\P^{t,\mu}_{t'}}_r)dr \ + \ \int_{t'}^s \sigma(r,X^{t,x,\mu}_r,\P^{t',\P^{t,\mu}_{t'}}_r)dW_r.
    \]
    We see that $(X_s^{t,x,\mu})_{s\in[t',T]}$ solves a classical stochastic differential equation, whose unique solution is $X^{t',X^{t,x,\mu}_{t'},\P^{t,\mu}_{t'}}$ (see Remark \ref{R: state processes}). Hence, $X^{t,x,\mu}$ and $X^{t',X^{t,x,\mu}_{t'},\P^{t,\mu}_{t'}}$ are indistinguishable on $[t',T]$.
\end{proof}
\begin{proof}[Proof of Theorem \ref{T:Snell}]
    By \eqref{snell envelope} we have
    \begin{align*}
         \Sc^{t,\xi}_s \ &= \ \esssup_{\tau\in\Tc_{s,T}}\E\big[G^{t,\xi}_\tau\big|\Fc_s\big] \ = \ \esssup_{\tau\in\Tc_{s,T}}\E\bigg[\int_{t}^{\tau}f(r,X^{t,\xi}_r,\P_{X^{t,\xi}_r})dr \ + \ g(X^{t,\xi}_\tau)\bigg|\Fc_s\bigg]\\
         &= \ \esssup_{\tau\in\Tc_{s,T}}\E\bigg[\int_{t}^{s}f(r,X^{t,\xi}_r,\P_{X^{t,\xi}_r})dr \ + \ \int_{s}^{\tau}f(r,X^{t,\xi}_r,\P_{X^{t,\xi}_r})dr \ + \ g(X^{t,\xi}_\tau)\bigg|\Fc_s\bigg].
    \end{align*}
    Since the first integral is $\Fc_s$-measurable and independent of $\tau$, we can write
    \[
        \Sc^{t,\xi}_s \ = \ \esssup_{\tau\in\Tc_{s,T}}\E\bigg[\int_{s}^{\tau}f(r,X^{t,\xi}_r,\P_{X^{t,\xi}_r})dr \ + \ g(X^{t,\xi}_\tau)\bigg|\Fc_s\bigg] \ + \ \int_{t}^{s}f(r,X^{t,\xi}_r,\P_{X^{t,\xi}_r})dr.
    \]
    Thus we only need to prove that
    \begin{equation}\label{Vtilde = esssup}
        \esssup_{\tau\in\Tc_{s,T}}\E\bigg[\int_{s}^{\tau}f(r,X^{t,\xi}_r,\P_{X^{t,\xi}_r})dr \ + \ g(X^{t,\xi}_\tau)\bigg|\Fc_s\bigg] \ = \ \Tilde{V}(s,X^{t,\xi}_s,\P_{X^{t,\xi}_s}).
    \end{equation}
    By Theorem \ref{T:representation V tilde}, we have
    \[
        \Tilde{V}(s,X^{t,\xi}_s,\P_{X^{t,\xi}_s}) \ = \ \esssup_{\tau\in\Tc_{s,T}}\E\bigg[\int_{s}^{\tau}f\Big(r,X^{s,X^{t,\xi}_s,\P_{X^{t,\xi}_s}}_r,\P^{s,\P_{X^{t,\xi}_s}}_r\Big)dr \ + \ g\Big(X^{s,X^{t,\xi}_s,\P_{X^{t,\xi}_s}}_\tau\Big)\bigg|\Fc_s\bigg].
    \]
    Now notice that, by the flow property (Lemma \ref{L: flow prop}), $\P$-a.s.,
        \[
            \Big(X^{s,X^{t,\xi}_s,\P_{X^{t,\xi}_s}}_r, \P^{s,\P_{X^{t,\xi}_s}}_r\Big)_{r\in[s,T]} \ = \ \Big(X^{t,\xi,\P_\xi}_r,\P^{t,\P_\xi}_r\Big)_{r\in[s,T]}.
        \]
    Moreover, by Remark \ref{R: state processes}, $X^{t,\xi,\P_\xi}=X^{t,\xi}$ a.s., and by Remark \ref{R: same marginals}, $\P^{t,\P_\xi}_r=\P_{X^{t,\xi}_r}$, for all $r\in[s,T]$. From these equalities we see that \eqref{Vtilde = esssup} holds. Finally, the characterization of the optimal stopping time $\hat\tau^{t,\xi}$ follows immediately from \eqref{optimal tau} and the expression of the Snell envelope in terms of $\tilde V$.
\end{proof}

%% file: dpp.tex
\section{Dynamic programming principle}\label{S:DPP}

As already mentioned, on the enlarged state space we recover a time-consistency property, in fact we are able to prove a dynamic programming principle for the extended value function $\tilde V$ in \eqref{V tilde}.

\begin{Theorem}[Dynamic Programming Principle]\label{T:DPP MKV 1}
    Let $t\in[0,T]$, $x\in\R^d$, $\mu\in\mathcal{P}_2(\R^d)$. Then, for every $t'\in[t,T]$,
    \begin{align}
        \tilde V(t,x,\mu) \ = \ \sup_{\tau\in\Tc_{t,T}} \E\bigg[\int_{t}^{\tau\land t'} f(s,X^{t,x,\mu}_s,&\P^{t,\mu}_s)ds \ + \ g(X^{t,x,\mu}_\tau)\mathbbm{1}_{\{\tau<t'\}}\notag \\
        &+ \ \tilde V(t',X^{t,x,\mu}_{t'},\P^{t,\mu}_{t'})\mathbbm{1}_{\{\tau\geq t'\}}\bigg]\label{DPP MKV}.
    \end{align}
\end{Theorem}
\begin{proof}
    For simplicity of notation, we define
    \[
        \Lambda(t,x,\mu) \coloneqq \sup_{\tau\in \Tc_{t,T}} \E\bigg[\int_{t}^{\tau\land t'} f(s,X^{t,x,\mu}_s,\P^{t,\mu}_s)ds + g(X^{t,x,\mu}_\tau)\mathbbm{1}_{\{\tau<t'\}} + \tilde V(t',X^{t,x,\mu}_{t'},\P^{t,\mu}_{t'})\mathbbm{1}_{\{\tau\geq t'\}}\bigg]. 
    \]
    We split the rest of the proof into two steps.

    \vspace{2mm}
    
    \noindent\textsc{Step} 1. \emph{Proof of the inequality} $\tilde V(t,x,\mu) \leq \Lambda(t,x,\mu)$. Let $\tau\in\Tc_{t,T}$ be fixed. By definition of $\Lambda$
    \[
        \Lambda(t,x,\mu) \ \geq \ \E\bigg[\int_{t}^{\tau\land t'} f(s,X^{t,x,\mu}_s,\P^{t,\mu}_s)ds + \ g(X^{t,x,\mu}_\tau)\mathbbm{1}_{\{\tau<t'\}} + \tilde V(t',X^{t,x,\mu}_{t'},\P^{t,\mu}_{t'})\mathbbm{1}_{\{\tau\geq t'\}}\bigg].
    \]
    On the other hand, by Theorem \ref{T:representation V tilde}, using that $\tau\vee t'\in\Tc_{t',T}$,
    \begin{align*}
        \tilde V(t',X^{t,x,\mu}_{t'},\P^{t,\mu}_{t'}) \ \geq\ \E\bigg[\int_{t'}^{\tau\vee t'} f(s,X^{t',X^{t,x,\mu}_{t'},\P^{t,\mu}_{t'}}_s,\P^{t',\P^{t,\mu}_{t'}}_s)ds\ + \ g(X^{t',X^{t,x,\mu}_{t'},\P^{t,\mu}_{t'}}_{\tau\vee t'})\bigg|\Fc_{t'}\bigg].
    \end{align*}
    By the flow property \eqref{eq prop flusso}, we have $(X^{t,x,\mu}_s,\P^{t,\mu}_s)_{s\in[t',T]}=(X^{t',X^{t,x,\mu}_{t'}, \P^{t,\mu}_{t'}}_s,\P^{t',\P^{t,\mu}_{t'}}_s)_{s\in[t',T]}$ a.s., so that
    \begin{align*}
        \tilde V(t',X^{t,x,\mu}_{t'},\P^{t,\mu}_{t'}) \ \geq\ \E\bigg[\int_{t'}^{\tau\vee t'} f(s,X^{t,x,\mu}_s,\P^{t,\mu}_s)ds \ + \ g(X^{t,x,\mu}_{\tau\vee t'})\bigg|\Fc_{t'}\bigg].
    \end{align*}
    Going back to the first inequality, we have (using that $\{\tau\geq t'\}$ is $\Fc_{t'}$-measurable)
    \begin{align}
        \Lambda(t,x,\mu) \ &\geq\ \E\bigg[\int_{t}^{\tau\land t'} 
        f(s,X^{t,x,\mu}_s,\P^{t,\mu}_s)ds \ + \ g(X^{t,x,\mu}_\tau)\mathbbm{1}_{\{\tau<t'\}} \notag \\
        &\hspace{2cm}+ \ \E\bigg[\int_{t'}^{\tau\vee t'} f(s,X^{t,x,\mu}_s,\P^{t,\mu}_s)ds \ + \ g(X^{t,x,\mu}_{\tau\vee t'})\bigg|\Fc_{t'}\bigg]\mathbbm{1}_{\{\tau\geq t'\}}\bigg] \notag \\
        &= \ \E\bigg[\int_{t}^{\tau\land t'} 
        f(s,X^{t,x,\mu}_s,\P^{t,\mu}_s)ds\ + \ \mathbbm{1}_{\{\tau\geq t'\}}\int_{t'}^{\tau\vee t'} f(s,X^{t,x,\mu}_s,\P^{t,\mu}_s)ds\notag\\
        &\hspace{5 cm}+ \ g(X^{t,x,\mu}_{\tau})\mathbbm{1}_{\{\tau<t'\}} \ + \ g(X^{t,x,\mu}_{\tau\vee t'})\mathbbm{1}_{\{\tau\geq t'\}}\bigg] \notag\\
        &= \ \E\bigg[\int_{t}^{\tau}f(s,X^{t,x,\mu}_s,\P^{t,\mu}_s)ds\ + \ g(X^{t,x,\mu}_\tau)\bigg] \ = \ \tilde J(t,x,\mu,\tau)\notag .
    \end{align}
    Thus, by the arbitrariness of $\tau\in\Tc_{t,T}$, we conclude that $\Lambda(t,x,\mu)\geq\tilde V(t,x,\mu)$.

    \vspace{2mm}
    
    \noindent\textsc{Step} 2. \emph{Proof of the inequality} $\tilde V(t,x,\mu) \geq \Lambda(t,x,\mu)$. Let $\varepsilon>0$, then there exists $ \tau_\varepsilon\in\Tc_{t,T}$ such that
    \[
        \Lambda(t,x,\mu) \ \leq\ \E\bigg[\int_{t}^{\tau_\varepsilon\land t'} f(s,X^{t,x,\mu}_s,\P^{t,\mu}_s)ds + g(X^{t,x,\mu}_{\tau_\varepsilon})\mathbbm{1}_{\{\tau_\varepsilon<t'\}} + \tilde V(t',X^{t,x,\mu}_{t'},\P^{t,\mu}_{t'})\mathbbm{1}_{\{\tau_\varepsilon\geq t'\}}\bigg] + \varepsilon.
    \]
    On the other hand, using the same $\varepsilon$, by Theorem \ref{T:representation V tilde} there exists $\tilde\tau_\varepsilon\in\Tc_{t',T}$ such that
    \[
        \tilde V(t',X^{t,x,\mu}_{t'},\P^{t,\mu}_{t'}) \ \leq \ \E\bigg[\int_{t'}^{\tilde\tau_\varepsilon} f(s,X^{t',X^{t,x,\mu}_{t'},\P^{t,\mu}_{t'}}_s,\P^{t',\P^{t,\mu}_{t'}}_s)ds + g(X^{t',X^{t,x,\mu}_{t'},\P^{t,\mu}_{t'}}_{\tilde\tau_\varepsilon})\bigg|\Fc_{t'}\bigg] + \varepsilon.
    \]
    Following the same reasoning as in \textsc{Step} 1, by the flow property we have that
    \[
        V(t',X^{t,x,\mu}_{t'},\P^{t,\mu}_{t'})\ \leq\ \E\bigg[\int_{t'}^{\tilde\tau_\varepsilon} f(s,X^{t,x,\mu}_s,\P^{t,\mu}_s)ds + g(X^{t,x,\mu}_{\tilde\tau_\varepsilon})\bigg|\Fc_{t'}\bigg] + \varepsilon.
    \]
    This implies that
    \begin{align*}
        \Lambda(t,x,\mu) \ &\leq\ \E\bigg[\int_{t}^{\tau_\varepsilon\land t'} f(s,X^{t,x,\mu}_s,\P^{t,\mu}_s)ds \ + \ g(X^{t,x,\mu}_{\tau_\varepsilon})\mathbbm{1}_{\{\tau_\varepsilon<t'\}}\\
        &\hspace{1.5cm} + \ \E\bigg[\int_{t'}^{\tilde\tau_\varepsilon} f(s,X^{t,x,\mu}_s,\P^{t,\mu}_s)ds \ + \ g(X^{t,x,\mu}_{\tilde\tau_\varepsilon})\bigg|\Fc_{t'}\bigg]\mathbbm{1}_{\{\tau_\varepsilon\geq t'\}}\bigg] \ + \ 2\varepsilon\\
        &= \ \E\bigg[\int_{t}^{\tau_\varepsilon\land t'} f(s,X^{t,x,\mu}_s,\P^{t,\mu}_s)ds \ + \ g(X^{t,x,\mu}_{\tau_\varepsilon})\mathbbm{1}_{\{\tau_\varepsilon<t'\}}\\
        &\hspace{1.5cm} + \ \bigg(\int_{t'}^{\tilde\tau_\varepsilon} f(s,X^{t,x,\mu}_s,\P^{t,\mu}_s)ds \ + \ g(X^{t,x,\mu}_{\tilde\tau_\varepsilon})\bigg)\mathbbm{1}_{\{\tau_\varepsilon\geq t'\}}\bigg]\ + \ 2\varepsilon.
    \end{align*}
    Now, let $\hat\tau_\varepsilon \coloneqq \tau_\varepsilon\mathbbm{1}_{\{\tau_\varepsilon<t'\}}+\tilde\tau_\varepsilon\mathbbm{1}_{\{\tau_\varepsilon\geq t'\}}$. Notice that $\hat\tau_\varepsilon\in\Tc_{t,T}$. As a matter of fact, we begin noting that $\hat\tau_\varepsilon$ takes values in $[t,T]$, since $\hat\tau_\varepsilon$ coincides with $\tau_\varepsilon$ or with $\tilde\tau_\varepsilon$ and they take values  in $[t,T]$ and $[t',T]$ respectively. Moreover, for every $s\in[t,T]$ it holds that
    \[
        \{\hat\tau_\varepsilon\leq s\} \ = \ \{\hat\tau_\varepsilon\leq s,\tau_\varepsilon< t'\}\cup\{\hat\tau_\varepsilon\leq s,\tau_\varepsilon\geq t'\} \ = \ \{\tau_\varepsilon\leq s,\tau_\varepsilon< t'\}\cup\{\tilde\tau_\varepsilon\leq s,\tau_\varepsilon\geq t'\}.
    \]
    Since $\tilde\tau_\varepsilon\geq t'$ a.s., if $s\in[t,t')$ it holds that
    \[
        \{\hat\tau_\varepsilon\leq s\} = \{\tau_\varepsilon\leq s\}\in\Fc_s.
    \]
    On the other hand, if $s\in[t',T]$, then
    \[
        \{\hat\tau_\varepsilon\leq s\} = \{\tau_\varepsilon< t'\}\cup(\{\tilde\tau_\varepsilon\leq s\}\cap\{\tau_\varepsilon\geq t'\})\in\Fc_s,
    \]
    where we used the following measurability properties:
    \begin{itemize}%[-]
        \item[-] $\{\tau_\varepsilon< t'\},\{\tau_\varepsilon\geq t'\}\in\Fc_{t'}\subset\Fc_s$ since $\F$ is right continuous and $\tau_\varepsilon\in\tilde \Tc_{t,T}$;
        \item[-] $\{\tilde\tau_\varepsilon\leq s\}\in\Fc_s$ since $\tilde\tau_\varepsilon\in \Tc_{t',T}$. 
    \end{itemize}
    Thus, we can conclude that $\hat\tau_\varepsilon$ is a stopping time, in particular $\hat\tau_\varepsilon\in\Tc_{t,T}$. Then, we can write
    \begin{align*}
        \Lambda(t,x,\mu) \ &\leq \ \E\bigg[\int_{t}^{\hat\tau_\varepsilon\land t'} f(s,X^{t,x,\mu}_s,\P^{t,\mu}_s)ds \ + \ g(X^{t,x,\mu}_{\hat\tau_\varepsilon})\mathbbm{1}_{\{\hat\tau_\varepsilon<t'\}}\\
        &\hspace{1.5cm} + \ \mathbbm{1}_{\{\hat\tau_\varepsilon \geq t'\}}\bigg(\int_{t'}^{\hat\tau_\varepsilon} f(s,X^{t,x,\mu}_s,\P^{t,\mu}_s)ds \ + \ g(X^{t,x,\mu}_{\hat\tau_\varepsilon})\bigg)\bigg]\ + \ 2\varepsilon.
    \end{align*}
    The latter inequality results from the following equalities:
    \begin{itemize}
        \item[-] $\hat\tau_\varepsilon\land t' \ = \
            \begin{cases}
                \tau_\varepsilon\land t', &\text{on }
                \{\tau_\varepsilon<t'\}\\
                \tilde\tau_\varepsilon\land t', &\text{on }\{\tau_\varepsilon\geq t'\}
            \end{cases}
            \ = \ \begin{cases}
                \tau_\varepsilon, &\text{on }
                \{\tau_\varepsilon<t'\}\\
                t', &\text{on }\{\tau_\varepsilon\geq t'\}
            \end{cases}
            \ = \ \tau_\varepsilon\land t';$
        \item[-] $\{\hat\tau_\varepsilon<t'\}=\{\hat\tau_\varepsilon<t',\tau_\varepsilon<t'\}\cup\{\hat\tau_\varepsilon<t',\tau_\varepsilon\geq t'\}=\{\tau_\varepsilon<t',\tau_\varepsilon<t'\}\cup\{\tilde\tau_\varepsilon<t',\tau_\varepsilon\geq t'\}$ $=\{\tau_\varepsilon<t'\};$
        \item[-] $\{\hat\tau_\varepsilon\geq t'\}=\{\hat\tau_\varepsilon \geq t',\tau_\varepsilon<t'\}\cup\{\hat\tau_\varepsilon\geq t',\tau_\varepsilon\geq t'\}=\{\tau_\varepsilon\geq t',\tau_\varepsilon<t'\}\cup\{\tilde\tau_\varepsilon\geq t',\tau_\varepsilon\geq t'\}$ $=\{\tau_\varepsilon\geq t'\}.$
    \end{itemize}
    As a consequence, we obtain
    \[
        \Lambda(t,x,\mu) \ \leq \ \tilde J(t,x,\mu,\hat\tau_\varepsilon) \ + \ 2\varepsilon \ \leq \ \tilde V(t,x,\mu) \ + \ 2\varepsilon.
    \]
    Finally, by the arbitrariness of $\varepsilon$, we conclude that $\Lambda(t,x,\mu)\leq\tilde V(t,x,\mu)$.
\end{proof}

%% file: hjb.tex
\section{Hamilton-Jacobi-Bellman equation}\label{S:HJB}

In the present section we prove that the extended value function $\tilde V$ in \eqref{V tilde} solves a suitable Hamilton-Jacobi-Bellman equation in the viscosity sense. Such an equation is derived relying on the dynamic programming principle (Theorem \ref{T:DPP MKV 1}) and it turns out to be the following variational inequality on $[0,T]\times\R^d\times\mathcal{P}_2(\R^d)$:
\begin{equation}\label{HJB MKV}
    \begin{cases}
    \vspace{2mm}\max\big\{\partial_t u + \Lc_t u + f, g - u\big\}=0, &\qquad \text{on }[0,T)\times\R^d\times\mathcal{P}_2(\R^d), \\
    u(T,x,\mu) = g(x), &\qquad (x,\mu)\in\R^d\times\mathcal{P}_2(\R^d),
\end{cases}
\end{equation}
where $\Lc_t$ is the second-order differential operator associated to $(X^{t,x,\mu},\P^{t,\mu})$ given by
\begin{align*}
    \Lc_t u(t,x,\mu) &= \langle b(t,x,\mu),\partial_x u(t,x,\mu)\rangle + \frac{1}{2}\text{tr}(\sigma\sigma^T (t,x,\mu)\partial^2_x u(t,x,\mu))\\
    &\quad + \int_{\R^d} \langle b(t,y,\mu), \partial_\mu u(t,x,\mu,y)\rangle\mu(dy) + \frac{1}{2}\int_{\R^d} \text{tr}(\sigma\sigma^T(t,y,\mu)\partial_x\partial_\mu u(t,x,\mu,y))\mu(dy).
\end{align*}
The derivatives with respect to the measure $\partial_\mu u$ and $\partial_x\partial_\mu u$ are the so-called $L$-derivatives or Lions derivatives, as defined for instance in \cite[Chapter 5]{CD18_I}. In particular, in the sequel we denote by $\Cc^{1,2,2}([0,T]\times\R^d\times\mathcal{P}_2(\R^d))$ the set of functions $\varphi\colon[0,T]\times\R^d\times\mathcal{P}_2(\R^d)\longrightarrow\R$ satisfying assumptions (A1)-(A2)-(A3) and (5.106) in \cite[Proposition 5.102]{CD18_I}, so that It\^o's formula (\cite[formula (5.107)]{CD18_I}) holds.

\begin{Definition}
    A continuous function $u\colon[0,T]\times\R^d\times\mathcal{P}_2(\R^d)\longrightarrow\R$ is a viscosity subsolution (resp. supersolution) to \eqref{HJB MKV} if:
    \begin{itemize}
        \item $u(T,x,\mu)\leq \,(\text{resp.}\geq)\ g(x)$, for every $(x,\mu)\in\R^d\times\P_2(\R^d)$;
        \item for every $(t,x,\mu)\in[0,T)\times\R^d\times\mathcal{P}_2(\R^d)$ and $\varphi\in\Cc^{1,2,2}([0,T]\times\R^d\times\mathcal{P}_2(\R^d))$ such that $u-\varphi$ has a global maximum (resp. minimum) at $(t,x,\mu)$ with value $0$, it holds that 
        \[
            \max\big\{\partial_t \varphi+ \Lc_t \varphi + f, g - \varphi\big\}\geq \,(\text{resp.} \leq)\ 0.
        \]
    \end{itemize}
    The function $u$ is a viscosity solution to \eqref{HJB MKV} if it is at the same time a viscosity subsolution and a viscosity supersolution.
\end{Definition}

\begin{Theorem}\label{T: V tilde solution HJB}
    Suppose that, in addition to \textnormal{Assumptions (\nameref{A1})} and \textnormal{(\nameref{A2})}, $b,\sigma, f$ are continuous on $[0,T]\times\R^d\times\mathcal{P}_2(\R^d)$. Then, the extended value function $\tilde V$ is a viscosity solution to the Hamilton-Jacobi-Bellman equation \eqref{HJB MKV}.
\end{Theorem}

\begin{proof} 
    We recall from Proposition \ref{P: Vtilde properties} that the function $\tilde V$ is continuous on $[0,T]\times\R^d\times\mathcal{P}_2(\R^d)$. Moreover, by definition $\Tilde{V}$ satisfies $\tilde V(T,x,\mu)=g(x)$, for every $(x,\mu)\in\R^d\times\mathcal{P}_2(\R^d)$. Then, it remains to prove the viscosity sub and supersolution properties on $[0,T)\times\R^d\times\mathcal{P}_2(\R^d)$.

    \vspace{2mm}
    
    \noindent\emph{Viscosity supersolution property.} Let $(t,x,\mu)\in[0,T)\times\R^d\times\mathcal{P}_2(\R^d)$ and $\varphi\in\Cc^{1,2,2}([0,T]\times\R^d\times\mathcal{P}_2(\R^d))$ be such that $\Tilde{V}-\varphi$ has a global minimum in $(t,x,\mu)$ with value $0$. Now, let $t'\coloneqq t+h$, with $h\in(0,T-t)$. By the dynamic programming principle \eqref{DPP MKV},
    \begin{align*}
        \tilde V(t,x,\mu) \ \geq \ \E\bigg[\int_{t}^{\tau\land (t+h)} f(s,X^{t,x,\mu}_s&,\P^{t,\mu}_s)ds \ + \ g(X^{t,x,\mu}_\tau)\mathbbm{1}_{\{\tau<t+h\}} + \\
        &+ \ \tilde V(t+h,X^{t,x,\mu}_{t+h},\P^{t,\mu}_{t+h})\mathbbm{1}_{\{\tau\geq t+h\}}\bigg],\quad\text{for all }\tau\in\Tc_{t,T}.
    \end{align*}
    Since $\tilde V(t,x,\mu)=\varphi(t,x,\mu)$ and $\tilde V\geq\varphi$, we get
    \begin{align*}
        \varphi(t,x,\mu) \ \geq \ \E\bigg[\int_{t}^{\tau\land (t+h)} f(s,X^{t,x,\mu}_s&,\P^{t,\mu}_s)ds \ + \ g(X^{t,x,\mu}_\tau)\mathbbm{1}_{\{\tau<t+h\}}\\
        &+ \ \varphi(t+h,X^{t,x,\mu}_{t+h},\P^{t,\mu}_{t+h})\mathbbm{1}_{\{\tau\geq t+h\}}\bigg],\quad\text{for all }\tau\in\Tc_{t,T}.
    \end{align*}
    By applying It\^o's formula (\cite[formula (5.107)]{CD18_I}) to $\varphi$ on the interval $[t,t+h]$, we obtain
    \begin{align*}
        &\varphi(t+h,X^{t,x,\mu}_{t+h},\P^{t,\mu}_{t+h})\\
        &= \varphi(t,x,\mu) + \int_t^{t+h} \partial_t \varphi(s,X^{t,x,\mu}_s,\P^{t,\mu}_s)ds + \int_t^{t+h} \langle b(s,X^{t,x,\mu}_s,\P^{t,\mu}_s),\partial_x \varphi(s,X^{t,x,\mu}_s,\P^{t,\mu}_s)\rangle ds \\
        &\quad + \frac{1}{2}\int_t^{t+h} \textup{tr}\big(\sigma\sigma^T(s,X^{t,x,\mu}_s,\P^{t,\mu}_s)\partial^2_x \varphi(s,X^{t,x,\mu}_s,\P^{t,\mu}_s)\big)ds\\
        &\quad + \int_t^{t+h}\bigg(\int_{\R^d} \langle b(s,y,\P^{t,\mu}_s),\partial_\mu \varphi(s,X^{t,x,\mu}_s,\P^{t,\mu}_s,y)\rangle\mu(dy)\bigg)ds\\
        &\quad + \frac{1}{2}\int_t^{t+h}\bigg(\int_{\R^d}\textup{tr}\big(\sigma\sigma^T(s,y,\P^{t,\mu}_s)\partial_y\partial_\mu \varphi(s,X^{t,x,\mu}_s,\P^{t,\mu}_s,y)\big)\mu(dy)\bigg)ds\\
        &\quad + \int_t^{t+h} \langle \partial_x \varphi(s,X^{t,x,\mu}_s,\P^{t,\mu}_s),\sigma(s,X^{t,x,\mu}_s,\P^{t,\mu}_s)dW_s\rangle\\
        &= \varphi(t,x,\mu) + \!\! \int_t^{t+h} \!\!\!\!\!\!\big(\partial_t\varphi + \Lc_t\varphi\big)(s,X^{t,x,\mu}_s,\P^{t,\mu}_s)ds + \!\! \int_t^{t+h} \!\!\!\!\!\!\langle \partial_x \varphi(s,X^{t,x,\mu}_s,\P^{t,\mu}_s),\sigma(s,X^{t,x,\mu}_s,\P^{t,\mu}_s)dW_s\rangle.
    \end{align*}
    Under our assumptions the stochastic integral is a martingale. Then, if we replace $\varphi(t+h,X^{t,x,\mu}_{t+h},\P^{t,\mu}_{t+h})$ in the previous inequality, we get that, for all $\tau\in\Tc_{t,T}$,
    \begin{align*}
        \varphi(t,x,\mu) \ \geq \ \E\bigg[&\int_{t}^{\tau\land (t+h)} f(s,X^{t,x,\mu}_s,\P^{t,\mu}_s)ds \ + \ g(X^{t,x,\mu}_\tau)\mathbbm{1}_{\{\tau<t+h\}}\\
        &+ \ \bigg(\varphi(t,x,\mu)\ + \ \int_t^{t+h} \big(\partial_t\tilde V + \Lc_t\varphi\big)(s,X^{t,x,\mu}_s,\P^{t,\mu}_s)ds\bigg)\mathbbm{1}_{\{\tau\geq t+h\}}\bigg].
    \end{align*}
    In particular, if we choose $\tau=t$, the inequality becomes
    \begin{equation}\label{step 1.1}
        \varphi(t,x,\mu) \ \geq \ \E[g(X^{t,x,\mu}_t)] \ = \ g(x).
    \end{equation}
    On the other hand, if we choose $\tau=t+h$, then
    \[
        \varphi(t,x,\mu) \ \geq \ \E\bigg[\int_t^{t+h}f(s,X^{t,x,\mu}_s,\P^{t,\mu}_s)ds + \varphi(t,x,\mu)+ \int_t^{t+h}\big(\partial_t\varphi+\Lc_t\varphi)(s,X^{t,x,\mu}_s,\P^{t,\mu}_s)ds\bigg],
    \]
    which can be written as
    \[
        0\ \geq \ \E\bigg[\int_t^{t+h}\big(\partial_t\varphi+\Lc_t\varphi + f\big)(s,X^{t,x,\mu}_s,\P^{t,\mu}_s)ds\bigg].
    \]
    Dividing the latter by $h$ and letting $h\to 0^+$, we find
    \begin{equation}\label{step 1.2}
        0\ \geq \ \big(\partial_t\varphi+\Lc_t\varphi + f\big)(t,x,\mu).
    \end{equation}
    Thus, since \eqref{step 1.1} and \eqref{step 1.2} hold simultaneously,
    \begin{equation*}
        \max\big\{(\partial_t\varphi + \Lc_t\varphi + f)(t,x,\mu), g(x) - \varphi(t,x,\mu)\big\} \ \leq \ 0.
    \end{equation*}
    \emph{Viscosity subsolution property.} Let $(t,x,\mu)\in[0,T)\times\R^d\times\mathcal{P}_2(\R^d)$  and $\varphi\in\Cc^{1,2,2}([0,T]\times\R^d\times\mathcal{P}_2(\R^d))$ be such that $\Tilde{V}-\varphi$ has a global maximum in $(t,x,\mu)$ with value $0$. To prove that
    \[
        \max\big\{(\partial_t \varphi + \Lc_t \varphi + f)(t,x,\mu), g(x) - \varphi(t,x,\mu)\big\} \ \geq \ 0,
    \]
    it suffices to show that at least one of the two terms is nonnegative. If $g(x) - \varphi(t,x,\mu)\geq 0$ the claim follows. Then we assume that $g(x) - \varphi(t,x,\mu)< 0$. Let $t'\coloneqq t+h$ with $h\in(0,T-t)$. By the dynamic programming principle \eqref{DPP MKV} there exists $\tau_h\in\Tc_{t,T}$ such that
    \begin{align*}
        \tilde V(t,x,\mu) \ \leq \ \E\bigg[\int_{t}^{\tau_h\land (t+h)} f(s,X^{t,x,\mu}_s,&\P^{t,\mu}_s)ds \ + \ g(X^{t,x,\mu}_{\tau_h})\mathbbm{1}_{\{{\tau_h}<t+h\}}\\
        &+ \ \tilde V(t+h,X^{t,x,\mu}_{t+h},\P^{t,\mu}_{t+h})\mathbbm{1}_{\{{\tau_h}\geq t+h\}}\bigg] \ + \ h^2.
    \end{align*}
    Since $\tilde V(t,x,\mu)=\varphi(t,x,\mu)$ and $\tilde V\leq\varphi$, we get
    \begin{align*}
        \varphi(t,x,\mu) \ \leq \ \E\bigg[\int_{t}^{\tau_h\land (t+h)} f(s,X^{t,x,\mu}_s,&\P^{t,\mu}_s)ds \ + \ g(X^{t,x,\mu}_{\tau_h})\mathbbm{1}_{\{{\tau_h}<t+h\}}\\
        &+ \ \varphi(t+h,X^{t,x,\mu}_{t+h},\P^{t,\mu}_{t+h})\mathbbm{1}_{\{{\tau_h}\geq t+h\}}\bigg] \ + \ h^2.
    \end{align*}
    By applying It\^o's formula (\cite[formula (5.107)]{CD18_I}) to $\varphi$ as in the proof of the viscosity supersolution property, we obtain
    \begin{align*}
        \varphi(t,x,\mu) \ &\leq \ \E\bigg[\int_{t}^{\tau_h\land (t+h)} f(s,X^{t,x,\mu}_s,\P^{t,\mu}_s)ds \ + \ g(X^{t,x,\mu}_{\tau_h})\mathbbm{1}_{\{{\tau_h}<t+h\}}\\
        &\quad\ + \ \bigg(\varphi(t,x,\mu) + \int_t^{t+h}\big(\partial_t\varphi + \Lc_t\varphi\big)(s,X^{t,x,\mu}_s,\P^{t,\mu}_s)ds\bigg)\mathbbm{1}_{\{{\tau_h}\geq t+h\}}\bigg] \ + \ h^2.
    \end{align*}
    This can be written as
    \begin{align}
        0 \ &\leq \ \E\bigg[\int_{t}^{\tau_h\land (t+h)} f(s,X^{t,x,\mu}_s,\P^{t,\mu}_s)ds \ + \  \mathbbm{1}_{\{{\tau_h}<t+h\}}\big(g(X^{t,x,\mu}_{\tau_h})-\varphi(t,x,\mu)\big)\notag\\
        &\hspace{4cm} + \ \mathbbm{1}_{\{{\tau_h}\geq t+h\}}\int_t^{t+h}\big(\partial_t\varphi + \Lc_t\varphi\big)(s,X^{t,x,\mu}_s,\P^{t,\mu}_s)ds\bigg] \ + \ h^2\notag\\
        &= \ \E\bigg[\mathbbm{1}_{\{{\tau_h}<t+h\}}\int_{t}^{\tau_h} f(s,X^{t,x,\mu}_s,\P^{t,\mu}_s)ds \ + \ \mathbbm{1}_{\{{\tau_h}<t+h\}}\big(g(X^{t,x,\mu}_{\tau_h})-\varphi(t,x,\mu)\big)\label{dpp+ito}\\
        &\hspace{3.3cm} + \ \mathbbm{1}_{\{{\tau_h}\geq t+h\}}\int_t^{t+h}\big(\partial_t\varphi + \Lc_t\varphi + f\big)(s,X^{t,x,\mu}_s,\P^{t,\mu}_s)ds\bigg] \ + \ h^2.\notag
    \end{align}
    Notice that
    \[
        \mathbbm{1}_{\{{\tau_h}\geq t+h\}}\int_t^{t+h}\big(\partial_t\varphi + \Lc_t\varphi + f\big)(s,X^{t,x,\mu}_s,\P^{t,\mu}_s)ds \ \longrightarrow \ 0\qquad\text{a.s.}\qquad\text{as }h\to0^+.
    \]
    Similarly
    \[
        \mathbbm{1}_{\{{\tau_h}<t+h\}}\int_{t}^{\tau_h} f(s,X^{t,x,\mu}_s,\P^{t,\mu}_s)ds \ \longrightarrow0 \ \qquad\text{a.s.}\qquad\text{as }h\to0^+.
    \] 
    Then, applying Fatou's lemma in \eqref{dpp+ito}, we have
    \[
    %\begin{equation}\label{liminf}
        0 \ \leq \ \liminf_{h\to 0^+}\E\big[\mathbbm{1}_{\{{\tau_h}<t+h\}}\big(g(X^{t,x,\mu}_{\tau_h})-\varphi(t,x,\mu)\big)\big].
    %\end{equation}
    \]
    Moreover, it holds that
    \begin{equation}\label{limsup}
        \limsup_{h\to0^+} \ \E\big[\mathbbm{1}_{\{{\tau_h}<t+h\}}\big(g(X^{t,x,\mu}_{\tau_h})-\varphi(t,x,\mu)\big)\big]\ \leq \ 0.
    \end{equation}
    Before proving \eqref{limsup}, we observe that, since $g$ is continuous and $g(x)-\varphi(t,x,\mu)<0$ by our assumption, there exists $R>0$ such that
    \[
        g(y) - \varphi(t,x,\mu) <0\qquad\text{for every }y\in\R^d\text{ with }|x-y|< R.
    \]
    Then, defining the stopping time 
    \[
        \theta \ \coloneqq \ \inf\{s\in[t,T]:|X^{t,x,\mu}_s -x|\geq R\},
    \]
    it holds that
    \begin{equation}\label{g-phi < 0}
        g(X^{t,x,\mu}_s) - \varphi(t,x,\mu)<0\qquad\text{ on }\{s<\theta\}.
    \end{equation}
    Hence (in the sequel we denote by $C$ a non-negative constant, which may change from line to line, independent of $t,x,\mu,h$, only depending on $g$, $b$, $\sigma$, $T$; moreover, $p',q',p\geq1$ and $1/p'+1/q'=1$)
    \begin{align*}
        &\E\Big[\mathbbm{1}_{\{{\tau_h}<t+h\}}\big(g(X^{t,x,\mu}_{\tau_h})-\varphi(t,x,\mu)\big)\Big] \ \leq \ \E\Big[\mathbbm{1}_{\{{\tau_h}<t+h\}}\big(g(X^{t,x,\mu}_{\tau_h})-\varphi(t,x,\mu)\big)^+\Big]\\
        %&= \ \E\Big[\mathbbm{1}_{\{{\tau_h}<t+h\}}\mathbbm{1}_{\{\tau_h<\theta\}}\big(g(X^{t,x,\mu}_{\tau_h})-\varphi(t,x,\mu)\big)^+\Big]\\
        %&\hspace{2cm} + \E\Big[\mathbbm{1}_{\{{\tau_h}<t+h\}}\mathbbm{1}_{\{\tau_h\geq \theta\}}\big(g(X^{t,x,\mu}_{\tau_h})-\varphi(t,x,\mu)\big)^+\Big]\\
        &\!\!\!\underset{\eqref{g-phi < 0}}{=}   \E\Big[\mathbbm{1}_{\{\theta\leq{\tau_h}<t+h\}}\big(g(X^{t,x,\mu}_{\tau_h})-\varphi(t,x,\mu)\big)^+\Big] \ \leq \ \E\Big[\mathbbm{1}_{\{\theta<t+h\}}\big|g(X^{t,x,\mu}_{\tau_h})-\varphi(t,x,\mu)\big|\Big]\\
        &\leq \ C\,\E\Big[\mathbbm{1}_{\{\theta<t+h\}}\big(1 + \sup_{s\in[t,T]}|X^{t,x,\mu}_s|^2+|\varphi(t,x,\mu)|\big)\Big]\\
        &\!\!\!\!\!\underset{\text{H\"older}}{\leq}  C\,\E\Big[\big(1 + \sup_{s\in[t,T]}|X^{t,x,\mu}_s|^2+|\varphi(t,x,\mu)|\big)^{q'}\Big]^\frac{1}{q'}\,\E\Big[\mathbbm{1}_{\{\theta<t+h\}}^{p'}\Big]^\frac{1}{p'}\\
        &\leq \ C\bigg(1 + \E\bigg[\sup_{s\in[t,T]}|X^{t,x,\mu}_s|^{2q'}\bigg]+|\varphi(t,x,\mu)|^{q'}\bigg)^\frac{1}{q'}\P(\theta< t+h)^\frac{1}{p'}\\
        &\!\!\!\underset{\eqref{new estimate MKV 1}}{\leq}  C\big(1 + |x|^{2q'}+\Norm{\mu}_2^{2q'}+|\varphi(t,x,\mu)|^{q'}\big)^\frac{1}{q'}\P(\theta< t+h)^\frac{1}{p'}\\
        &\leq \ C\big(1 + |x|^2+\Norm{\mu}_2^2+|\varphi(t,x,\mu)|\big)\,\P(\theta< t+h)^\frac{1}{p'} \\
        &= \ C\big(1 + |x|^2+\Norm{\mu}_2^2+|\varphi(t,x,\mu)|\big)\,\P\bigg(\sup_{s\in[t,t+h]}|X^{t,x,\mu}_s-x|\geq R\bigg)^\frac{1}{p'}\\
        &\!\!\!\!\!\underset{\text{Markov}}{\leq} C\big(1 + |x|^2+\Norm{\mu}_2^2+|\varphi(t,x,\mu)|\big)\, \frac{\E\bigg[\sup_{s\in[t,t+h]}|X^{t,x,\mu}_s-x|^p\bigg]^\frac{1}{p'}}{R^\frac{p}{p'}}\\
        &\leq \ C\big(1 + |x|^2+\Norm{\mu}_2^2+|\varphi(t,x,\mu)|\big)\,h^\frac{p}{2p'}\,\big(1 + |x|^{p}+\Norm{\mu}_2^{p}\big)^\frac{1}{p'},
    \end{align*}
    where the last inequality follows from the standard estimate $\E\big[\sup_{s\in[t,t+h]}|X_s^{t,x,\mu} - x\big|^p]\leq Ch^\frac{p}{2}(1 + |x|^p + \Norm{\mu}_2^p)$. Notice that the estimates above imply \eqref{limsup}. Thus, we have
    \begin{equation}\label{lim}
        \lim_{h\to0^+}\E\big[\mathbbm{1}_{\{{\tau_h}<t+h\}}\big(g(X^{t,x,\mu}_{\tau_h})-\varphi(t,x,\mu)\big)\big]\ = \ 0.
    \end{equation}
    This, in turn, implies
    \begin{equation}\label{conv l1 indicatrice}
        \lim_{h\to0^+}\E\big[\mathbbm{1}_{\{{\tau_h}<t+h\}}\big] = 0.
    \end{equation}
    Indeed,
    \begin{align*}
        \E\big[\mathbbm{1}_{\{{\tau_h}<t+h\}}\big]\big(g(x)-\varphi(t,x,\mu)\big)\ &= \ \E\big[\mathbbm{1}_{\{{\tau_h}<t+h\}}\big(g(X^{t,x,\mu}_{\tau_h})-\varphi(t,x,\mu)\big)\big]\\
        &\quad \ - \ \E\big[\mathbbm{1}_{\{{\tau_h}<t+h\}}\big(g(X^{t,x,\mu}_{\tau_h}) - g(x)\big)\big],
    \end{align*}
    where the first term on the right tends to 0 by \eqref{lim}, while the second term tends to 0 by the continuity of $g$, the fact that $X^{t,x,\mu}$ has continuous trajectories, and by Lebesgue's dominated convergence theorem (using the sub-quadratic growth of $g$). Then, since $g(x)-\varphi(t,x,\mu) < 0$, necessarily \eqref{conv l1 indicatrice} holds.\\ 
    So, up to consider a subsequence,
    \[
        \mathbbm{1}_{\{{\tau_h}<t+h\}} \ \longrightarrow \ 0\qquad\text{a.s.}\qquad\text{as }h\to0^+
    \]
    and then
    \[
        \mathbbm{1}_{\{{\tau_h}\geq t+h\}} \ \longrightarrow \ 1\qquad\text{a.s.}\qquad\text{as }h\to0^+.
    \]
    Now, going back to \eqref{dpp+ito} and dividing by $h$, we get
    \begin{align}
        0 \ &\leq \ \E\bigg[\mathbbm{1}_{\{{\tau_h}<t+h\}}\frac{1}{h}\int_{t}^{\tau_h} f(s,X^{t,x,\mu}_s,\P^{t,\mu}_s)ds \ + \ \mathbbm{1}_{\{{\tau_h}<t+h\}}\frac{1}{h}\big(g(X^{t,x,\mu}_{\tau_h})-\varphi(t,x,\mu)\big)\label{dpp+ito 2}\\
        &\hspace{4 cm} + \ \mathbbm{1}_{\{{\tau_h}\geq t+h\}}\frac{1}{h}\int_t^{t+h}\big(\partial_t\varphi + \Lc_t\varphi + f\big)(s,X^{t,x,\mu}_s,\P^{t,\mu}_s)ds\bigg] \ + \ h\notag.
    \end{align}
    Thanks to the convergences proved above and to the mean value theorem, it holds that
    \[
        \mathbbm{1}_{\{{\tau_h}\geq t+h\}}\frac{1}{h}\int_t^{t+h}\big(\partial_t\varphi + \Lc_t\varphi + f\big)(s,X^{t,x,\mu}_s,\P^{t,\mu}_s)ds \ \overset{h\to0^+}{\longrightarrow} \ \big(\partial_t\varphi + \Lc_t\varphi + f\big)(t,x,\mu)\quad\text{a.s.}
    \]
    On the other hand, by the convergence of the indicator function and the continuity of $f$, $(X_s^{t,x,\mu})_s$, $(\P^{t,\mu}_s)_s$, we obtain
    \begin{align*}
        \bigg|\mathbbm{1}_{\{{\tau_h}<t+h\}}\frac{1}{h}\int_{t}^{\tau_h} f(s,X^{t,x,\mu}_s,\P^{t,\mu}_s)ds\bigg|\ &\leq \ \mathbbm{1}_{\{{\tau_h}<t+h\}}\frac{1}{h}\int_{t}^{\tau_h} \big|f(s,X^{t,x,\mu}_s,\P^{t,\mu}_s)\big|ds\\
        &\leq \ \mathbbm{1}_{\{{\tau_h}<t+h\}}\max_{s\in[t,T]}|\bar{f_s}|\frac{\tau_h-t}{h}\\
        &\leq \ \mathbbm{1}_{\{{\tau_h}<t+h\}}\max_{s\in[t,T]}|\bar{f_s}|\longrightarrow0\qquad\text{a.s.}\qquad\text{as }h\to0^+,
    \end{align*}
    where $\bar{f_s}:=f(s,X^{t,x,\mu}_s,\P^{t,\mu}_s)$. In particular, the last inequality is due to the fact that $\tau_h-t<h$ on $\{{\tau_h}<t+h\}$. Therefore, by passing to the limit, \eqref{dpp+ito 2} takes the form
    \[
        0 \ \leq \ \limsup_{h\to0^+}\E\bigg[\mathbbm{1}_{\{{\tau_h}<t+h\}}\frac{1}{h}\big( g(X^{t,x,\mu}_{\tau_h})-\varphi(t,x,\mu)\big)\bigg]\ + \ \big(\partial_t\varphi + \Lc_t\varphi + f\big)(t,x,\mu).
    \]
    Finally, define 
    \[
        \ell \ := \ \limsup_{h\to0^+} \ \E\bigg[\mathbbm{1}_{\{{\tau_h}<t+h\}}\frac{1}{h}\big( g(X^{t,x,\mu}_{\tau_h})-\varphi(t,x,\mu)\big)\bigg]
    \]
    %%da cambiare assolutamente
    With a similar argument to the one used to prove \eqref{limsup}, we get that $\ell\leq0$. %Indeed,
    Hence we can conclude that
    \[
    \big(\partial_t\varphi + \Lc_t\varphi + f\big)(t,x,\mu)\ \geq \ - \ell \ \geq \ 0.
    \]
\end{proof}